\documentclass[onefignum,onetabnum]{siamonline190516}
\pdfoutput=1

\usepackage[c2]{optidef}
\usepackage[autostyle]{csquotes}
\usepackage[textsize=scriptsize]{todonotes}
\usepackage{biblatex}
\usepackage[font=footnotesize]{caption}
\usepackage{subcaption}
\usepackage{tabularx}
\usepackage[margin=0.9in]{geometry}

\usepackage{algorithm}
\usepackage{algorithmic}
\usepackage{booktabs}
\usepackage{csvsimple}
\usepackage{siunitx,arydshln}

\newcommand{\R}{\mathbb{R}}
\newcommand{\Kbb}{\mathbb{K}}

\newcommand{\Bbold}{\mathbb{B}}
\newcommand{\Ebb}{\mathbb{E}}

\newcommand{\Ical}{\mathcal{I}}
\newcommand{\Jcal}{\mathcal{J}}
\newcommand{\Acal}{\mathcal{A}}

\newcommand{\Bcal}{\mathcal{B}}

\newcommand{\Dcal}{\mathcal{D}}
\newcommand{\Lcal}{\mathcal{L}}
\newcommand{\Tcal}{\mathcal{T}}

\newcommand{\Ecal}{\mathcal{E}}

\DeclareMathOperator*{\argmin}{arg\,min}
\newcommand{\scalar}[2]{\left \langle #1 , #2 \right \rangle}
\newcommand{\frob}[2]{\left \langle #1 , #2 \right \rangle_F}
\DeclareMathOperator{\minimize}{minimize}
\DeclareMathOperator{\rank}{rank}
\DeclareMathOperator{\spa}{span}

\usepackage{lipsum}
\usepackage{amsfonts, amssymb}
\usepackage{bm}
\usepackage{graphicx}
\usepackage{epstopdf}
\ifpdf
  \DeclareGraphicsExtensions{.eps,.pdf,.png,.jpg}
\else
  \DeclareGraphicsExtensions{.eps}
\fi

\usepackage{enumitem}
\setlist[enumerate]{leftmargin=.5in}
\setlist[itemize]{leftmargin=.5in}

\newsiamremark{remark}{Remark}
\newsiamremark{hypothesis}{Hypothesis}
\newsiamremark{assumption}{Assumption}
\crefname{hypothesis}{Hypothesis}{Hypotheses}
\newsiamthm{claim}{Claim}
\allowdisplaybreaks

\headers{Bilevel Imaging Learning Problems as MPCC}{J.C. De los Reyes}

\title{Bilevel Imaging Learning Problems as Mathematical Programs with\\ Complementarity Constraints: Reformulation and Theory
}

\author{Juan Carlos De los Reyes\thanks{Research Center for Mathematical Modeling (MODEMAT), Escuela Polit\'ecnica Nacional, Quito, Ecuador
  (\email{juan.delosreyes@epn.edu.ec}, \url{http://www.modemat.epn.edu.ec/}).}
}

\usepackage{amsopn}
\DeclareMathOperator{\diag}{diag}
\DeclareMathOperator{\Cos}{Cos}
\DeclareMathOperator{\Sin}{Sin}
\DeclareMathOperator{\Tan}{Tan}

\ifpdf
\hypersetup{
  pdftitle={Bilevel Imaging Learning Problems as Mathematical Programs with Complementarity Constraints},
  pdfauthor={J.C. De los Reyes}
}
\fi
\begin{document}
\maketitle

\begin{abstract}
    We investigate a family of bilevel imaging learning problems where the lower-level instance corresponds to a convex variational model involving first- and second-order nonsmooth sparsity-based regularizers. By using geometric properties of the primal-dual reformulation of the lower-level problem and introducing suitable auxiliar variables, we are able to reformulate the original bilevel problems as \emph{Mathematical Programs with Complementarity Constraints (MPCC)}. For the latter, we prove tight constraint qualification conditions (MPCC-RCPLD and partial MPCC-LICQ) and derive Mordukhovich (M-) and Strong (S-) stationarity conditions. The stationarity systems for the MPCC turn also into stationarity conditions for the original formulation. Second-order sufficient optimality conditions are derived as well, together with a local uniqueness result for stationary points. The proposed reformulation may be extended to problems in function spaces, leading to MPCC's with constraints on the gradient of the state. The MPCC reformulation also leads to the efficient use of available large-scale nonlinear programming solvers, as shown in a companion paper, where different imaging applications are studied.
  \end{abstract}

\begin{keywords}
  Bilevel optimization, variational models, machine learning, mathematical programs with complementarity constraints.
\end{keywords}

\begin{AMS}
  49K99, 90C33, 68U10, 68T99, 65K10
\end{AMS}

\section{Introduction}

Bilevel imaging learning problems were introduced in \cite{tappen2007utilizing} for learning Markov random fields models, and in \cite{de2013image,kunisch2013bilevel} for optimally learning noise models and nonsmooth sparsity-based regularizers in variational denoising problems. Thereafter, several other imaging applications have been successfully considered within this framework, e.g., mixed noise models \cite{calatroni2013dynamic,calatroni2019analysis}, higher-order regularizers \cite{reyes2015a,davoli2018one,davoli2019adaptive,hintermuller2017optimal}, blind deconvolution problems \cite{hintermuller2015bilevel}, nonlocal models \cite{d2021bilevel,bartels2020parameter}.

The main difficulty of these problems relies on the nonsmooth structure of the lower-level instances, which prevented the application of standard bilevel programming theory.
In classical bilevel optimization, whenever the lower-level cost function is convex and differentiable and no additional inequality constraints are present, the problems may equivalently be written as \emph{Mathematical Programs with Complementarity Constraints (MPCC)}, which allows the use of the rich analytical toolbox developed for MPCC in the last decades to derive sharp optimality conditions \cite{luo1996mathematical,flegel2005guignard,outrata2000generalized}. Since in most of the cases, however, this equivalence does not hold \cite{dempe2002foundations}, much of the research carried out in this field consists precisely in proving under which circumstances an MPCC formulation is possible.


In the case of bilevel imaging learning problems with total variation, for instance, due to the difficulties already mentioned, optimality conditions have been previously obtained using a local regularization of the nonsmooth terms and performing an asymptotic analysis thereafter \cite{de2013image,van2017learning}, yielding a \emph{C-stationarity} system (see also \cite{hintermuller2017analytical,hintermuller2019generating} for a related approach based on a dual reformulation of the lower-level problem). Alternatively, a direct nonsmooth approach was considered in \cite{hintermuller2015bilevel} and \cite{dlrvillacis2021} to learn point spread functions in blind deconvolution models and the weight in front of the fidelity term in denoising models, respectively. In both such cases, the parameter affects the fidelity term and, based on variational analysis tools, \emph{M-stationarity} systems were  derived. Recently \cite{delosreyesvillacis}, M-stationarity conditions were also obtained for total variation bilevel learning problems, when the scale-dependent parameter appears within the regularizer. In summary, so far, M-stationarity systems are the sharpest ones that have been obtained for total variation bilevel problems.

In this paper we extend and improve previous results, providing sharper optimality conditions for bilevel imaging learning problems by means of an MPCC reformulation. By restricting our attention to first- and second-order sparsity-based convex regularizers, and exploiting the geometric nature of the primal-dual reformulation of the lower-level problem, we are indeed able, after introducing suitable auxiliar variables, to reformulate the bilevel instances as MPCC. This reformulation opens the door to a detailed characterization of stationarity conditions; we are able to demonstrate \emph{M-} and \emph{S-stationarity} (\Cref{thm: strong stationarity MPCC} and \Cref{thm: strongg stationarity general problem}) under suitable assumptions. Moreover, also second-order sufficient optimality conditions may be derived in this manner (\Cref{thm: SSC}) and an infinite-dimensional bilevel counterpart may be stated as well (\Cref{sec:infinite}).

The verification that optimal parameters for variational imaging models are M- or S-stationary points leads to very important consequences, both theoretical and numerical. Theoretically, relevant MPCC properties, such as minimality, local uniqueness or stability, can only be obtained for this type of points through first- and second-order conditions. Numerically, the characterization of M- and S-stationary points enables the use of efficient algorithms, with local superlinear convergence rates, to compute them \cite{luo1996mathematical}. The results obtained in this article are, therefore, of importance to conclude that optimal parameters for variational models are robust with respect to changes in the data (e.g., noisy image, training set) and that their computation can be carried out efficiently with existing computational software to solve large-scale MPCC.


The outline of the paper is as follows. We will present the general bilevel problem and its MPCC reformulation in \Cref{sec:problem_formulation}, illustrating different particular regularizers such as \emph{Total Variation} and \emph{Second-Order Total Generalized Variation}. The detailed analysis of the bilevel problems with total variation will be presented in \Cref{sec:stationarity_conditions}, where MPCC-RCPLD and partial MPCC-LICQ will be verified, and corresponding M- and S-stationarity systems derived. An extension of the obtained reformulation to the infinite-dimensional setting will be briefly explored in \Cref{sec:infinite}, outlining its main difficulties. In \Cref{sec:second}, second-order sufficient optimality conditions are studied and the local uniqueness of minima is theoretically verified. The general bilevel learning problem with first- and second-order nonsmoth sparsity-based convex regularizers is analyzed in \Cref{sec:general}, where MPCC-RCPLD is verified and a corresponding M-stationarity system derived. Finally, in \Cref{sec:conclusions} we draw some conclusions and outline some extensions of the obtained results to other application fields.

\section{Problem Statement and Reformulation}\label{sec:problem_formulation}
In this work we are concerned with bilevel learning problems of the following form:
\begin{mini!}
    {\substack{(u,\lambda,\sigma,\alpha,\beta)}}{\Jcal (u)}{\label{eq:bilevel_problem}}{}
    \addConstraint{u = \argmin_{v\in\R^d} \, \Ecal(v,\lambda,\sigma, \alpha, \beta)}{}{}
    \addConstraint{P(\lambda),R(\sigma),Q(\alpha), S(\beta) \geq 0,}{}{}
\end{mini!}
where the lower-level energy is given by
\begin{multline}\label{eq:patch_abstract_lower_level}
    \Ecal(v,\lambda,\sigma, \alpha, \beta) := \Dcal_0(v) + \sum_{j=1}^K \sum_{i=1}^{k_j} P_j(\lambda_j)_i \Dcal_j(v)_{i} + \sum_{j=1}^L \sum_{i=1}^{l_j} R_j(\sigma_j)_i |(\mathbb B_j v)_i |\\ +
    \sum_{j=1}^M \sum_{i=1}^{m_j} Q_j(\alpha_j)_i \|(\mathbb K_j v)_i\|+
    \sum_{j=1}^N \sum_{i=1}^{n_j} S_j(\beta_j)_i \|(\mathbb E_j v)_i\|_F,
\end{multline}
with real vector parameters $\lambda_j ,\sigma_j , \alpha_j, \beta_j$. The operators $P_j:\R^{|\lambda_j|} \mapsto \R^{k_j}$, $R_j: \R^{|\sigma_j|} \mapsto \R^{l_j}$, $Q_j:\R^{|\alpha_j|} \mapsto \R^{m_j}$ and $S_j: \R^{|\beta_j|} \mapsto \R^{n_j}$
are assumed to be twice continuously differentiable, while the operators inside the norms, $\mathbb B_j: \R^d \mapsto \R^{l_j}$, $\mathbb K_j: \R^d \mapsto \R^{m_j \times 2}$ and $\mathbb E_j: \R^d \mapsto (\textrm{Sym}_2)^{n_j}$, are assumed to be linear, where $\textrm{Sym}_2$ denotes the space of $2 \times 2$ symmetric matrices. The functions $\Dcal_0:\R^d \to\R$ and $\Dcal_j:\R^d \to\R^{k_j},~j=1, \dots, K$, correspond to different \textit{data fidelity terms} considered in the model. The adjoint of a linear operator $T$ will be denoted by $T^\star$. The notation $|\cdot|, ~\|\cdot \|$ and
$\|\cdot\|_F$ stands for the absolute value, the Euclidean norm and the Frobenius norm, respectively. Similarly, $\scalar{\cdot}{\cdot}$
and $\frob{\cdot}{\cdot}$ denote the Euclidean and Frobenius scalar products, respectively. For a given matrix $A \in \R^{n \times m}$ and an index set $I \subset \{1,\dots,n\}$, we denote by $A_I \in \R^{|I| \times m}$ the submatrix formed by the rows of $A$ indexed by $I$.

The energy \eqref{eq:patch_abstract_lower_level} considered in this manuscript encompasses, therefore, the main first and second order sparsity-based convex regularizers used in variational image processing. Moreover, the regularizing operators may involve local filter kernels as well.

Assuming existence of a training set of clean and noisy images or some information about noise statistics, the upper-level loss function $\Jcal$ incorporates this information. This occurs, for instance, through a quadratic loss using ground-truth images \cite{de2013image}, quality measures aimed at preserving jumps \cite{reyes2015a} or deviation of image residuals from an estimated variance corridor \cite{hintermuller2017optimal}.


Next, we will reformulate problem \cref{eq:bilevel_problem} using duality properties of the corresponding lower-level problem. To do so, we assume along the paper the following condition on the different data fidelity terms.

\begin{assumption} \label{assump:D_convex_diff}
  The data fidelity functions $\Dcal_j$, $j = 0, \dots, K$, are convex in each component and twice continuously differentiable. Moreover, the data fidelity function $\Dcal_0$ is strongly convex.
\end{assumption}

We start by studying two particular instances of \eqref{eq:bilevel_problem} before introducing the general problem.

\subsection{Bilevel Total Variation}
A typical lower-level problem is the classical Rudin-Osher-Fatemi denoising model given by the energy minimizer:
\begin{equation} \label{eq: TV lower level}
  u= \argmin_{v \in \R^d} \quad \Dcal_0(v;f) + \sum_{i=1}^{m} Q(\alpha)_i\|(\mathbb K v)_i\|,
\end{equation}
where $\Kbb : \R^d \to \R^{m\times 2}$ stands for the discrete gradient operator and $f$ is the image corrupted with noise. In this case the choice of regularizer is known as \textit{isotropic total variation}, which has been widely adopted in the image processing community due to its edge-preserving properties
(see, e.g., \cite{scherzer2010handbook} and the references therein).

Given the assumptions on the fidelity terms, described in \cref{assump:D_convex_diff}, we can guarantee existence of a unique minimizer for problem \cref{eq: TV lower level}. Moreover, its necessary and sufficient optimality condition is given by the following variational inequality of the second kind \cite{de2015numerical}:
\begin{equation}\label{eq:lower_level_vi}
    \scalar{\nabla_u \Dcal_0(u;f)}{v-u} + \sum_{i=1}^{m} Q(\alpha)_i \|(\Kbb v)_i\| - \sum_{i=1}^{m} Q(\alpha)_i \|(\Kbb u)_i\| \ge 0, \quad \forall v \in \R^d.
\end{equation}
By using Fenchel duality theory \cite{ekeland1999convex}, it is possible to  write an equivalent primal-dual optimality condition for the lower level problem \cref{eq:lower_level_vi} as follows:
\begin{subequations}\label{eq:primal_dual_stationarity_lower_level}
    \begin{align}
        \nabla_u \Dcal_0(u;f) + \Kbb^\star q &= 0,\\
        \scalar{q_i}{(\Kbb u)_i}&= Q(\alpha)_i \|(\Kbb u)_i\|, && i=1,\dots,m,\\
        \|q_i\|&\le Q(\alpha)_i, && i=1,\dots,m,
    \end{align}
\end{subequations}
where $q$ stands for the dual multiplier and $\Kbb^\star q := K_x^\intercal q^{x} + K_y^\intercal q^{y}$ corresponds to the discrete divergence operator. Along the paper we assume that all row vectors of the matrices $K_x$ and $K_y$ are different from zero, and that $\rank (K_x)= \rank(K_y)=m$, with $m \leq d$.

Now, by replacing the lower-level problem with the optimality condition \cref{eq:primal_dual_stationarity_lower_level}, we get a new single level optimization problem given by:
\begin{mini!}
    {\substack{u,\lambda,\alpha}}{\Jcal(u)}{\label{eq:bilevel_problem_kkt_reformulation}}{}
    \addConstraint{\nabla_u \Dcal_0(u;f) + \Kbb^\star q}{= 0}
    \addConstraint{\scalar{q_i}{(\Kbb u)_i}}{= Q(\alpha)_i \|(\Kbb u)_i\|,}{\quad i=1,\dots,m\label{eq:primal_dual_problem_con_2}}
    \addConstraint{\|q_i\|}{\le Q(\alpha)_i,}{\quad i=1,\dots,m.\label{eq:primal_dual_problem_con_3}}
\end{mini!}
This alternative optimization problem involves the nonstandard complementarity constraint \cref{eq:primal_dual_problem_con_2}, which resembles a componentwise cosinus formula. To gain further insight, let us reformulate, using a trigonometric change of variables, the components of the dual variable and the gradient of the primal variable, $q_i$ and $(\Kbb u)_i$, as follows:
\begin{equation*}
  q_i = \delta_i [\cos(\phi_i),\sin(\phi_i)]^\intercal, \qquad
  (\Kbb u)_i = r_i[\cos(\theta_i),\sin(\theta_i)]^\intercal,
\end{equation*}
and let us also introduce the following inactive, active and biactive index sets:
\begin{align*}
    & \Ical(u):=\{i\in\{1,\dots,m\}: \; (\Kbb u)_i \neq 0\}, \\
    & \Acal(u):=\{i\in\{1,\dots,m\}: \; (\Kbb u)_i = 0, \|q_i\| < Q(\alpha)_i\}, \\
    & \Bcal(u):=\{i\in\{1,\dots,m\}: \; (\Kbb u)_i = 0, \|q_i\|=Q(\alpha)_i\}.
\end{align*}
In the index set $\Acal(u) \cup \Bcal(u)$, \Cref{eq:primal_dual_problem_con_2} holds trivially as $(\Kbb u)_i=0$ in this set. For the case $i\in\Ical(u)$ the equality holds only if both vectors $q_i$ and
$(\Kbb u)_i$ are colinear. Moreover, in this case the dual variable can be uniquely determined, which is not the case in the active and biactive sets, where it is contained in the unit ball (see \Cref{fig:vectors_by_index_set} for a sketch on this).
\begin{figure}[ht]
    \centering
    \begin{subfigure}{0.31\textwidth}
        \centering
        \begin{tikzpicture}
          \draw[step=0.5cm,gray,very thin] (-1.8,-1.8)grid(1.8,1.8);
          \draw(-2.0,0)-- (2.0,0);
          \draw(0,-2.0)-- (0,2.0);
          \draw(0,0)circle[radius=1.5cm];
          \draw[->,color=red,very thick](0,0)--(1.0606,1.0606) node[below] {$q_i$};
          \draw[->,color=blue,thick](0,0)--(1.5,1.5) node[left] {$(\Kbb u)_i$};
      \end{tikzpicture}      
        \caption{$i\in\Ical(u)$}
        \label{fig:sub-first}
    \end{subfigure}
    \begin{subfigure}{0.31\textwidth}
        \centering
        \begin{tikzpicture}
          \draw[step=0.5cm,gray,very thin] (-1.8,-1.8)grid(1.8,1.8);
          \draw(-2.0,0)-- (2.0,0);
          \draw(0,-2.0)-- (0,2.0);
          \draw(0,0)circle[radius=1.5cm];
          \fill[color=red,opacity=0.4] (0,0) circle (1.45cm);
          \fill[color=blue] (0,0) circle (2pt) node[below] {$(\Kbb u)_i$};
          \draw[->,color=red,thick](0,0)--(0.75,0.77) node[right] {$q_i$};
      \end{tikzpicture}
              \caption{$i\in\Acal(u)$}
        \label{fig:sub-second}
    \end{subfigure}
    \begin{subfigure}{0.31\textwidth}
        \centering
        \begin{tikzpicture}
          \draw[step=0.5cm,gray,very thin] (-1.8,-1.8)grid(1.8,1.8);
          \draw(-2.0,0)-- (2.0,0);
          \draw(0,-2.0)-- (0,2.0);
          \draw[very thick,color=red](0,0)circle[radius=1.5cm];
          \fill[color=blue] (0,0) circle (2pt) node[below] {$(\Kbb u)_i$};;
          \draw[->,color=red,thick](0,0)--(1.0606,1.0606) node[right] {$q_i$};
      \end{tikzpicture}
              \caption{$i\in\Bcal(u)$}
        \label{fig:sub-third}
    \end{subfigure}
    \caption{Geometric insight of the primal-dual system for the different index sets.}
    \label{fig:vectors_by_index_set}
\end{figure}
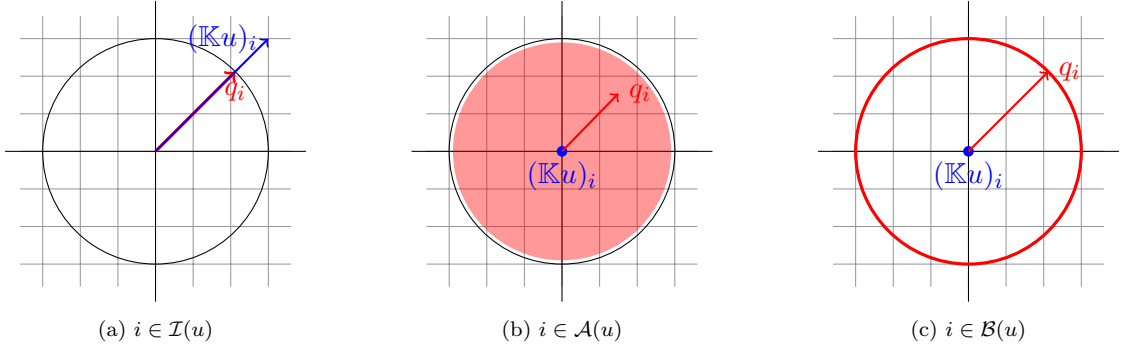

This insight allows us to conclude that the angles are the same, i.e., $\theta_i = \phi_i$, for all $i=1, \dots, m$, which enables the reformulation of the terms in \Cref{eq:primal_dual_problem_con_2} as follows:
\begin{align*}
    & \scalar{q_i}{(\Kbb u)_i} = r_i\cos(\theta_i)\delta_i\cos(\theta_i) + r_i\sin(\theta_i)\delta_i\sin(\theta_i)
    =r_i\delta_i(\cos^2(\theta_i)+\sin^2(\theta_i))=r_i\delta_i,\\
    & Q(\alpha)_i \|(\Kbb y)_i\| = Q(\alpha)_i (r_i^2\cos^2(\theta_i)+r_i^2\sin^2(\theta_i))^{1/2} = Q(\alpha)_i |r_i|,
\end{align*}
yielding
\begin{equation*}
    r_i\delta_i - Q(\alpha)_i |r_i| = 0,\qquad \forall i=1,\dots,m.
\end{equation*}
The inequality constraints \cref{eq:primal_dual_problem_con_3} can then be rewritten as
\begin{equation*}
    \|q_i\| = (\delta_i^2\cos^2(\phi_i)+\delta_i^2\sin^2(\phi_i))^{1/2} = |\delta_i| \le Q(\alpha)_i.
\end{equation*}
It is clear that whenever $r_i=\delta_i=0$, the angle $\theta_i$ may be arbitrarily chosen as it does not affect the representation of the primal and dual variables.
Furthermore, we know that $r_i \delta_i\ge 0$ and $\delta_i\ge 0$, implying $0\le Q(\alpha)_i r_i = r_i\delta_i$ and $0 \le \delta_i \le Q(\alpha)_i$.

Altogether, we can rewrite problem \cref{eq:bilevel_problem_kkt_reformulation} as the following MPCC:
\begin{mini}
    {\substack{u,q,\alpha,r,\delta,\theta}}{\Jcal (u)}{\label{eq:bilevel_mpec_formulation}}{}
    \addConstraint{\nabla_u \Dcal_0(u;f)+ \Kbb^\star q}{= 0}
    \addConstraint{(\Kbb u)_i}{= r_i [\cos(\theta_i),\sin(\theta_i)]^\intercal,}{\;\quad i=1,\dots,m}
    \addConstraint{q_i}{= \delta_i [\cos(\theta_i),\sin(\theta_i)]^\intercal,}{\;\quad i=1,\dots,m}
    \addConstraint{\delta_i}{\ge 0,}{\;\quad i=1,\dots,m}
    \addConstraint{0\le r_i}{\perp (Q(\alpha)_i -\delta_i) \ge 0,}{\;\quad i=1,\dots,m.}
\end{mini}
Problem \cref{eq:bilevel_mpec_formulation} has classical equality and inequality constraints, along with standard complementarity constraints (see, e.g., \cite{scheel2000mathematical}).

\subsection{Bilevel Second Order Total Generalized Variation}
Higher order regularizers were introduced as a remedy to some shortcomings of first order ones, such as isotropic or anisotropic total variation. Arguably the most well-known artifact introduced by TV is the so-called staircasing effect, which leads to a piecewise constant reconstruction of smooth intensity variations in the image. Indeed, one possibility to counteract such artifacts is the introduction of higher-order derivatives in the image regularization terms. Chambolle and Lions \cite{chambolle1997image}, for instance, proposed a higher-order method by means of an infimal convolution of the total variation of the image and the total variation of the image gradient, called \emph{Infimal Convolution Total Variation (ICTV)}. Other ways to combine first- and second-order regularizers where introduced, for instance, by Chan et al. \cite{chan2000highorder}, who consider total variation minimization together with weighted versions of the Laplacian, by Masnou and Morel \cite{masnou1998level}, who introduced the Euler-elastica functional that combines total variation regularization with curvature penalization, by Parisotto et al. \cite{parisotto2020higher}, who proposed a higher-order anisotropic version of total variation, among other.

In \cite{bredies2010total}, Bredies et al. proposed \emph{Total Generalized Variation (TGV)} as higher-order variants of TV. This family of regularizers has gained popularity thanks to its frequently superior performance compared with TV and ICTV, and has been studied in depth both theoretically and numerically \cite{knoll2011second,martin2013tgvdti,hintermuller2017analytical}. Bilevel problems with TGV lower-level instances have also been investigated in the last years, yielding results on existence and approximability of optimal parameters \cite{reyes2015a}, first-order necessary optimality conditions for scalar and scale-dependent parameters \cite{van2017learning,calatroni2017bilevel}, dualization approaches \cite{hintermuller2017optimal,hintermuller2019generating}, and iterative optimization algorithms \cite{calatroni2017bilevel,reyes2015a}.

In the case of the \emph{Second Order Total Generalized Variation ($TGV^2$)}, the lower-level denoising problem is given by
\begin{mini!}
  {u=(v,w)}{\Dcal_0(v;f) + \sum_{i=1}^{m} Q(\alpha)_i \|(\mathbb K v-w)_i\| + \sum_{i=1}^{n} S(\beta)_i \|(\mathbb{\hat E} w)_i\|_F,}{\label{eq: TGV lower level}}{}
\end{mini!}
where $\Kbb$ stands for the discrete gradient operator and $\mathbb{\hat E}$ for the discrete symmetrized gradient tensor. Thanks to the convexity of the energy, and by using Fenchel duality theory, a necessary and sufficient optimality condition for the denoising problem is given by the primal-dual system:
\begin{subequations}\label{eq:primal_dual_stationarity_lower_level_TGV}
    \begin{align}
        & \nabla_v \Dcal_0(v;f) + \Kbb^\star q = 0,\\
        & q_i   = -\mathbb{\hat E}^\star \hat \Lambda_i, 
        && i=1,\dots,m,\\
        & \scalar{q_i}{(\Kbb v -w)_i} = Q(\alpha)_i \|(\Kbb v-w)_i\|, &&i=1,\dots,m,\\
        & \|q_i\| \le Q(\alpha)_i, && i=1,\dots,m\\
        & \frob{\hat \Lambda_i}{(\mathbb{\hat E} w)_i} = S(\beta)_i \|(\mathbb{\hat E} w)_i\|_F, &&i=1,\dots,n, \label{eq:primal_dual_stationarity_lower_level_TGV_collinear_matrices} \\
        & \|\hat \Lambda_i\|_F \le S(\beta)_i, && i=1,\dots,n. \label{eq:primal_dual_stationarity_lower_level_TGV_bound_matrix}
    \end{align}
\end{subequations}
Since each $q_i$ and each $(\Kbb v-w)_i$ are elements of $\R^2$, we proceed as in the TV case and use the change of variables
\begin{align*}
  q_i = \delta_i[\cos(\theta_i),\sin(\theta_i)]^\intercal, \qquad  (\Kbb v-w)_i = r_i [\cos(\theta_i),\sin(\theta_i)]^\intercal,
\end{align*}
which implies that
\begin{equation*}
  \|q_i\| = (\delta_i^2\cos^2(\theta_i)+\delta_i^2\sin^2(\theta_i))^{1/2} = |\delta_i| \le Q(\alpha)_i.
\end{equation*}
Also
\begin{align*}
    \scalar{q_i}{(\Kbb v-w)_i}
    &=r_i\delta_i(\cos^2(\theta_i)+\sin^2(\theta_i))=r_i\delta_i,\\
    \|(\Kbb v-w)_i\| &= (r_i^2\cos^2(\theta_i)+r_i^2\sin^2(\theta_i))^{1/2} = |r_i|,
\end{align*}
which yields
\begin{equation*}
    r_i\delta_i - Q(\alpha)_i  |r_i| = 0,\quad \forall i=1,\dots,m.
\end{equation*}

Concerning the second-order term,
we obtain, thanks to the symmetry of each matrix 
$$\hat \Lambda_i = \begin{bmatrix}
    \lambda_{1,1}^i & \lambda_{1,2}^i\\
    \lambda_{2,1}^i & \lambda_{2,2}^i
\end{bmatrix},$$ 
that their Frobenius norms are given by $\|\hat \Lambda_i\|_F = \sqrt{(\lambda_{1,1}^i)^2+ 2(\lambda_{1,2}^i)^2+ (\lambda_{2,2}^i)^2}$, which turn out to be equivalent to the Euclidean norm of the transformed vectors
${\Lambda}_i = [\lambda_{1,1}^i, \sqrt{2} \lambda_{1,2}^i,\lambda_{2,2}^i]^\intercal.$
In a similar manner, we introduce the notation $$({\mathbb E} w)_i := \left[ (\mathbb{\hat E} w)^i_{1,1}, \sqrt{2} (\mathbb{\hat E} w)^i_{1,2}, (\mathbb{\hat E} w)^i_{2,2} \right]^\intercal.$$
As a consequence, we may transform the matrix constraints to vector ones and obtain
\begin{subequations} \label{eq:primal_dual_stationarity_lower_level_TGV_collinear_vectors}
  \begin{align}
    & \scalar{{\Lambda}_i}{({\mathbb E} w)_i} = S(\beta)_i \|({\mathbb E} w)_i\|, &&i=1,\dots,n, \label{eq:primal_dual_stationarity_lower_level_TGV_collinear_vectors_1}\\
    & \| \Lambda_i\| \le S(\beta)_i, && i=1,\dots,n. \label{eq:primal_dual_stationarity_lower_level_TGV_collinear_vectors_2}
  \end{align}
\end{subequations}

\Cref{eq:primal_dual_stationarity_lower_level_TGV_collinear_vectors_1} then implies collinearity of the vectors $\Lambda_i$ and $({\mathbb E} w)_i$, for each $i$. Hence, using a \emph{spherical coordinates} change of variables, we get the representation
\begin{align*}
    ({\mathbb E} w)_i &= \rho_i \left[  \sin(\phi_i) \cos(\varphi_i), \sin(\phi_i) \sin(\varphi_i), \cos(\phi_i)\right]^\intercal,\\
     \Lambda_i &= \tau_i [ \sin(\phi_i) \cos(\varphi_i), \sin(\phi_i) \sin(\varphi_i), \cos(\phi_i)]^\intercal,
\end{align*}
which implies that
\begin{align*}
  \scalar{{\Lambda}_i}{({\mathbb E} w)_i}
  & = \rho_i \tau_i (\sin^2(\phi_i) \cos^2(\varphi_i) + \sin^2(\phi_i) \sin^2(\varphi_i) + \cos^2(\phi_i))= \rho_i \tau_i,
\end{align*}
and
\begin{align*}
  \|{\Lambda}_i \| = |\tau_i| \sqrt{\sin^2(\phi_i) \cos^2(\varphi_i) + \sin^2(\phi_i) \sin^2(\varphi_i) + \cos^2(\phi_i)} = |\tau_i|.
\end{align*}
Consequently, system \eqref{eq:primal_dual_stationarity_lower_level_TGV_collinear_vectors} may be rewritten as
\begin{align*}
  \rho_i \tau_i = S(\beta)_i |\rho_i| \qquad \text{and} \qquad |\tau_i| \leq S(\beta)_i, \qquad \forall i=1,\dots,n,
\end{align*}
respectively.

Altogether, we arrive at the following equivalent MPCC reformulation of the second-order TGV bilevel problem:
\begin{mini}< b >
    {\substack{v, w, q, {\Lambda}, \alpha, \beta, r, \delta , \rho, \tau, \theta, \phi, \varphi}}{\Jcal (v)}{\label{eq:bilevel_mpec_formulation_TGV}}{}
    \addConstraint{\nabla_v \Dcal_0(v;f) + \Kbb^\star q}{= 0}
    \addConstraint{q_i}{= -{\mathbb E}^\star \Lambda_i}{\;\quad i=1,\dots,m}
    \addConstraint{(\Kbb v-w)_i}{= r_i [\cos(\theta_i),\sin(\theta_i)]^\intercal,}{\;\quad i=1,\dots,m}
    \addConstraint{q_i}{= \delta_i [\cos(\theta_i),\sin(\theta_i)]^\intercal,}{\;\quad i=1,\dots,m}
    \addConstraint{({\mathbb E} w)_i}{= \rho_i \left[ \sin(\phi_i) \cos(\varphi_i), \sin(\phi_i) \sin(\varphi_i), \cos(\phi_i)\right]^\intercal,}{\; \quad i=1\dots,n}
    \addConstraint{\Lambda_i}{=\tau_i [\sin(\phi_i) \cos(\varphi_i), \sin(\phi_i) \sin(\varphi_i), \cos(\phi_i)]^\intercal,}{\; \quad i=1\dots,n}
    \addConstraint{\delta_i}{\geq 0,}{\;\quad i=1\dots,m}
    \addConstraint{\tau_i}{\ge 0,}{\;\quad i=1\dots,n}
    \addConstraint{0\le r_i \perp (Q(\alpha)_i-\delta_i)}{\ge 0,}{\;\quad i=1\dots,m}
    \addConstraint{0\le \rho_i \perp (S(\beta)_i-\tau_i)}{\ge 0,}{\;\quad i=1\dots,n.}
\end{mini}

\subsection{General Bilevel Problem}\label{sec:general_problem}
For the general problem \eqref{eq:bilevel_problem}, a primal-dual reformulation of the lower-level instance may be carried out in a similar manner as for the TV and TGV cases. The additional difficulty is related with the absolute value terms $|(\mathbb B_j v)_i |$, which we handle also using duality. 

Indeed, introducing the dual variables $c_{j,i}$ we get the extremality conditions
\begin{subequations}
  \begin{align}
    c_{j,i} (\mathbb B_j u)_i &= R_j(\sigma_j)_i |(\mathbb B_j u)_i|, && j=1, \dots, L, \, i=1, \dots,l_j,\label{eq: abs value reformulation 1}\\
    |c_{j,i}| & \leq R_j(\sigma_j)_i, && j=1, \dots, L, \, i=1, \dots,l_j. \label{eq: abs value reformulation 2}
  \end{align}    
\end{subequations}
The constraints \eqref{eq: abs value reformulation 2} are clearly box ones for each $c_{j,i}$. For the constraints \eqref{eq: abs value reformulation 1}, due to the positivity of $R_j(\sigma_j)_i$ and the Cauchy-Schwarz inequality, one direction always holds. Therefore, the constraints can be formulated as
\begin{equation*}
  c_{j,i} (\mathbb B_j u)_i \geq R_j(\sigma_j)_i |(\mathbb B_j u)_i|, \qquad j=1, \dots, L, \, i=1, \dots,l_j.
\end{equation*}
Splitting the absolute value, we get equivalently
\begin{align*}
  -c_{j,i} (\mathbb B_j u)_i &\leq  R_j(\sigma_j)_i (\mathbb B_j u)_i \leq c_{j,i} (\mathbb B_j u)_i, && j=1, \dots, L, \, i=1, \dots,l_j,\\
  c_{j,i} (\mathbb B_j u)_i & \geq 0, && j=1, \dots, L, \, i=1, \dots,l_j.
\end{align*}


Thus, using bilateral constraints for the absolute value terms and changes of variables for the Euclidean and the Frobenius norm terms, as in \Cref{eq:bilevel_mpec_formulation} and \Cref{eq:bilevel_mpec_formulation_TGV}, respectively, the MPCC reformulation of the general bilevel learning problem is given by:
\begin{subequations} \label{eq:bilevel_mpec_formulation_general}
  \begin{equation}
    \underset{u, c, q, \Lambda, \lambda, \sigma, \alpha, \beta, r, \delta, \rho, \tau, \theta, \phi, \varphi}{\minimize} \Jcal (u)
  \end{equation}
  subject to the primal-dual denoising problem:
  \begin{equation} \label{eq:bilevel_mpec_formulation_general_first_constraint}
    \nabla_u \Dcal_0(u) + \sum_{j=1}^K \sum_{i=1}^{k_j} P_j(\lambda_j)_i \nabla_u (\Dcal_j(u)_{i})+ 
    \sum_{j=1}^L \mathbb B_j^\intercal c_j +
    \sum_{j=1}^M \mathbb K_j^\star q_j+
    \sum_{j=1}^N \mathbb E_j^\star \Lambda_j
    = 0,
  \end{equation}
  the changes of variables of primal and dual variables of the Euclidean and Frobenius norm terms:
  \begin{align}
    (\Kbb_j u)_i &= r_{j,i} [\cos(\theta_{j,i}),\sin(\theta_{j,i})]^\intercal, && j=1,\dots,M, i=1,\dots,m_j\\
    q_{j,i} &= \delta_{j,i} [\cos(\theta_{j,i}),\sin(\theta_{j,i})]^\intercal, && j=1,\dots,M, i=1,\dots,m_j\\
    (\mathbb E_j u)_i &= \rho_{j,i} \left[ \sin(\phi_{j,i}) \cos(\varphi_{j,i}), \sin(\phi_{j,i}) \sin(\varphi_{j,i}), \cos(\phi_{j,i})\right]^\intercal, && j=1,\dots,N, i=1,\dots,n_j\\
    \Lambda_{j,i} &=\tau_{j,i} \left[ \sin(\phi_{j,i}) \cos(\varphi_{j,i}), \sin(\phi_{j,i}) \sin(\varphi_{j,i}), \cos(\phi_{j,i}) \right]^\intercal, && j=1,\dots,N, i=1,\dots,n_j
  \end{align}
  the inequality constraints for the absolute value terms:
  \begin{align}
    -c_{j,i} (\mathbb B_j u)_i &\leq  R_j(\sigma_j)_i (\mathbb B_j u)_i \leq c_{j,i} (\mathbb B_j u)_i, && j=1, \dots, L, \, i=1, \dots,l_j,\\
    -R_j(\sigma_j)_i &\leq c_{j,i} \leq R_j(\sigma_j)_i, && j=1, \dots, L, \, i=1, \dots,l_j,\\
    0 & \leq c_{j,i} (\mathbb B_j u)_i , && j=1, \dots, L, \, i=1, \dots,l_j,
  \end{align}
  the positivity constraints
  \begin{align}
    R_j(\sigma_j)_i &\geq 0, && j=1, \dots, L, \, i=1, \dots,l_j\\
    P_j(\lambda_j)_i & \geq 0, && j=2, \dots, L, \, i=1, \dots,k_j,\\
    \delta_{j,i}, &\geq 0, && j=1, \dots, M, \, i=1, \dots,m_j,\\
    \tau_{j,i}, & \geq 0, && j=1, \dots, N, \, i=1, \dots,n_j,
  \end{align}
  and the complementarity constraints
  \begin{align}
    0 &\le r_{j,i} \perp (Q_j(\alpha_j)_i-\delta_{j,i})\ge 0, && j=1, \dots, M, \, i=1, \dots,m_j,\\
    0 &\le \rho_{j,i} \perp (S_j(\beta_j)_i-\tau_{j,i})\ge 0, && j=1, \dots, N, \, i=1, \dots,n_j. \label{eq:bilevel_mpec_formulation_general_final_constraint}
  \end{align}
\end{subequations}

\section{Stationarity Conditions}\label{sec:stationarity_conditions}
For mathematical programs with complementarity constraints, it is of particular importance to characterize local optimal solutions by means of an optimality system as sharp as possible. To do so, taylored constraint qualification conditions are required to hold in order to get existence of Lagrange multipiers and establish sign-conditions on the so-called biactive set.

Consider the general MPCC given by:
\begin{mini!}
  {\substack{x \in \mathbb R^n}}{\ell(x)}{\label{eq:MPCC}}{}
  \addConstraint{h(x)}{= 0}
  \addConstraint{g(x)}{\leq0}
  \addConstraint{0\leq M(x)}{\perp N(x)\geq 0,}
\end{mini!}
where $\ell: \R^n \to \R$, $h:\R^n \to \R^p$, $g:\R^n \to \R^m$ and $M,N:\R^n \to \R^q$ are continuously differentiable functions. We denote the feasible set by $$X=\{x\in\R^n: g(x)\leq 0,h(x)=0, 0 \leq M(x) \perp N(x) \geq 0\}$$ and introduce the following index sets for a given feasible point $x\in X$:
\begin{align*}
  \Acal(x)&{}:=\{i:M_i(x)=0,N_i(x)>0\},\\
  \Bcal(x)&{}:=\{i:M_i(x)=0,N_i(x)=0\},\\
  \Ical(x)&{}:=\{i:M_i(x)>0,N_i(x)=0\}.
\end{align*}
The sets $\Acal(x)$, $\Ical(x)$ and $\Bcal(x)$ are called \emph{active, inactive} and \emph{biactive} sets, respectively. Moreover, for a given index set, where a set of relations hold, we introduce for simplicity the short notation
$\{ e(x)_i =0\}:=\{i \in \{1, \dots, n\}: e(x)_i =0\}.$

To prove existence of Lagrange multipliers for \eqref{eq:MPCC} a suitable constraint qualification condition has to be satisfied. For this type of problems, however, it can be shown that standard nonlinear programming constraint qualification conditions such as \emph{Linear Independence Constraint Qualification (LICQ)} or \emph{Mangasarian-Fromovitz Constraint Qualification (MFCQ)} do not hold, even for very simple problem instances (see, e.g., \cite{dlrvillacis2021}). Therefore, tailor-made constraint qualification conditions have been proposed in last years for MPCC, along with different notions of stationarity such as \emph{Clarke-stationarity}, \emph{Mordukovich-stationarity} and \emph{Strong-stationarity} \cite{flegel2005guignard,flegel2007optimality}.

To formulate different MPCC constraint qualification conditions, let us start by defining the tangent cone $\Tcal_X(x)$ and the MPCC-linearized tangent cone $\Lcal_X^{MPCC}(x)$.
\begin{definition}
  Let $x^*\in X$ be feasible for \eqref{eq:MPCC}.
  \begin{itemize}
    \item The tangent (Bouligand) cone to a set $X$ at the point $x^*$ is given by
    \begin{equation*}
        \Tcal_X(x^*):=\{d\in\R^n:\;\exists x_k\xrightarrow{X} x^*,\; \exists t_k \downarrow 0:\;t_k^{-1} (x_k-x^*)\to d\}.
    \end{equation*}
    \item The MPCC-Linearized Tangent Cone at $x^*\in X$ is given by
        \begin{align*}
            L_X^{MPCC}(x^*):=\{d\in\R^n:\;&\nabla g_i(x^*)^\intercal d\le 0,&&\;\forall i\in \{g_i(x^*)=0 \},\\
            &\nabla h_i(x^*)^\intercal d = 0,&&\;\forall i=1,\dots,p,\\
            &\nabla M_i(x^*)^\intercal d = 0,&&\;\forall i\in\Acal(x^*),\\
            &\nabla N_i(x^*)^\intercal d = 0,&&\;\forall i\in\Ical(x^*),\\
            &\nabla M_i(x^*)^\intercal d \ge 0,&&\;\forall i\in\Bcal(x^*),\\
            &\nabla N_i(x^*)^\intercal d \ge 0,&&\;\forall i\in\Bcal(x^*),\\
            &(\nabla M_i(x^*)^\intercal d)(\nabla N_i(x^*)^\intercal d) = 0,&&\;\forall i\in\Bcal(x^*)\}.
        \end{align*}
  \end{itemize}
\end{definition}

Depending on the relation between these two cones and/or their polars, different MPCC constraint qualification conditions may be established.
\begin{definition}
Let $x^*\in X$ be feasible for \eqref{eq:MPCC}.
\begin{itemize}
\item[a)] MPCC-Abadie Constraint Qualification (MPCC-ACQ) holds at $x^*$ if
    \begin{equation*}
        \Tcal_X(x^*) = L_X^{MPCC}(x^*).
    \end{equation*}
\item[b)] MPCC-Relaxed Constant Positive Linear Dependence Condition (MPCC-RCPLD). Let $I_1 \subset \{1, \dots, p\}$ be such that $\{\nabla h_i(x^*)\}_{i \in I_1}$ is a basis for $\operatorname{span} \{\nabla h_i(x^*)\}_{i=1, \dots,p}$. MPCC-RCPLD holds at $x^*$ if there is a neighborhood $V(x^*)$ of $x^*$ such that
\begin{itemize}
  \item[i)] $\{\nabla h_i(x)\}_{i=1, \dots,p}$ has the same rank for every $x \in V(x^*)$;
  \item[ii)] For any $I_2 \subset \{i: g_i(x^*)=0 \}$, $I_3 \subset \Acal(x^*) \cup \Bcal(x^*)$ and $I_4 \subset \Ical(x^*) \cup \Bcal(x^*)$, whenever there exist multipliers, not all zero,
  $(\lambda, \mu, \gamma, \nu)$ with $\mu_i \geq 0$ for each $i \in I_1$, either $\gamma_j \nu_j =0 $ or $\gamma_j > 0$, $\nu_j > 0$ for each $j \in \Bcal(x^*)$, such that
  \begin{equation*}
      \sum_{i \in I_1} \lambda_i \nabla h_i(x^*) + \sum_{i \in I_2} \mu_{i} \nabla g_{i}(x^*) - \sum_{i \in I_3} \gamma_i \nabla M_i(x^*) - \sum_{i \in I_4} \nu_i \nabla N_i(x^*) = 0,
  \end{equation*}
  then the vectors
  \begin{equation*}
      \nabla h_i(x), ~i \in I_1, \quad
      \nabla g_i(x), ~i \in I_2, \quad
      \nabla M_i(x), ~i \in I_3, \quad
      \nabla N_i(x), ~ i \in I_4
  \end{equation*}
  are linearly dependent for any $x \in V(x^*)$.
\end{itemize}

\item[d)] MPCC-Generalized Mangasarian Fromowitz Constraint Qualification (MPCC-GMFCQ) holds at $x^*$ if there is no nonzero vector  $(\lambda,\mu,\gamma,\nu)$ such that
    \begin{align*}
        & \sum_{i=1}^p \lambda_i \nabla h_i(x^*) + \sum_{\{g_i(x^*)=0 \}} \mu_i \nabla g_i(x^*) - \sum_{i=1}^q \gamma_i \nabla M_i(x^*) - \sum_{i=1}^q \nu_i \nabla N_i(x^*) = 0,\\
        & \mu_i \geq 0, ~ \forall i \in \{g_i(x^*)=0 \}, \quad \gamma_i = 0, ~\forall i \in \Ical(x^*),\quad \nu_i = 0, ~\forall i \in \Acal(x^*),\\
        & \text{either } \gamma_i>0, ~ \nu_i>0 \quad \text{or} \quad \nu_i \gamma_i=0, \quad \forall i \in \Bcal(x^*).
    \end{align*}
\item[e)] Partial MPCC-Linear Independence Constraint Qualification holds at $x^*$ if
    \begin{equation*}
        \sum_{i=1}^p \lambda_i \nabla h_i(x^*) + \sum_{\{g_i(x^*)=0 \}} \mu_i \nabla g_i(x^*) - \sum_{\Acal(x^*) \cup \Bcal(x^*)} \gamma_i \nabla M_i(x^*) - \sum_{\Ical(x^*) \cup \Bcal(x^*)} \nu_i \nabla N_i(x^*) = 0
    \end{equation*}
    implies that $\gamma_i=0$ and $\nu_i =0$, for all $i \in \Bcal(x^*)$.
\end{itemize}
\end{definition}
The MPCC-RCPLD is one of the weakest verifiable constraint qualification conditions that leads to M-stationarity (see \cite{guo2013second} and the references therein).
Moreover, the following implications hold for the different qualification conditions introduced above (see, e.g., \cite{guo2013second} ): MPCC-GMFC $\Rightarrow$ MPCC-RCPLD $\Rightarrow$ MPCC-ACQ.

A feasible point $x^* \in X$ is called \textbf{C-stationary} (\textbf{C} for Clarke) for \eqref{eq:MPCC} if there exist multipliers $\lambda \in \R^p, \mu \in \R^m$ and $\gamma, \nu \in \R^q$, such that the following system is satisfied:
\begin{subequations} \label{eq:clarke_stationarity_system}
\begin{align}
    \nabla \ell(x^*) + \nabla h(x^*)\lambda + \nabla g(x^*) \mu - \nabla M(x^*) \gamma - \nabla N(x^*)\nu & = 0,\\
    0 \ge g(x^*)\perp \mu & \ge 0,\\
    h(x^*) & = 0,\\
    M(x^*) & \ge 0,\\
    N(x^*) & \ge 0,\\
    \nu_i &=0,\; &&\forall i\in \Acal(x^*),\\
    \gamma_i &= 0, \;&&\forall i\in \Ical(x^*),\\
    \gamma_i \nu_i &\ge 0,\;&&\forall i\in \Bcal(x^*).
\end{align}
\end{subequations}
If the last condition in \eqref{eq:clarke_stationarity_system} is replaced by the sharper characterization:
\begin{equation*}
    \gamma_i\nu_i=0 \quad \vee \quad \gamma_i\ge 0,\nu_i \ge 0,\qquad \text{for all } i\in\Bcal(x^*),
\end{equation*}
then the system is called \textbf{M-stationary} (\textbf{M} for Mordukhovich), which holds under MPCC-ACQ \cite{JaneYe2005}, and the corresponding multiplier $(\lambda, \mu, \gamma, \nu)$ is called \underline{M-multiplier}.
The sharpest optimality condition is called \textbf{S-stationary} (\textbf{S} for strong), where, in addition to \eqref{eq:clarke_stationarity_system}, the following sign condition on the biactive set holds:
\begin{equation*}
    \gamma_i\ge 0,\nu_i \ge 0,\qquad \text{for all } i\in\Bcal(x^*).
\end{equation*}
The corresponding multiplier is called \underline{S-multiplier}.

\begin{theorem}[Flegel, Kanzow \cite{flegel2005guignard}] \label{them: strong stationarity}
Let $x^* \in X$ be a local optimal solution of problem \eqref{eq:MPCC}. If both MPCC-ACQ and partial MPCC-LICQ hold at $x^*$, then $x^*$ is S-stationary.
\end{theorem}

The strong stationarity system is the sharpest possible set of relations that characterize minima of MPCC problems. In the case of an empty biactive set, all stationarity concepts presented previously coincide.

\subsection{Bilevel Total Variation}
Concerning bilevel imaging learning problems with total variation, Clarke stationarity has been previously obtained in \cite{de2013image}, and Mordukovich stationarity was proved for blind point deconvolution in \cite{hintermuller2015bilevel} and for denoising problems in \cite{dlrvillacis2021,delosreyesvillacis}. Next, we will introduce an alternative technique for proving \emph{M-stationarity} and will investigate under which conditions \emph{strong stationarity} also holds for this family of problems.

Let us recall the reformulated bilevel learning problem with total variation:
\begin{mini!}
    {\substack{x=(u,q, \alpha,r,\delta, \theta)}}{\Jcal (u)}{\label{eq:bilevel_mpec_formulationTV}}{}
    \addConstraint{\nabla_u \Dcal_0(u;f)+ \Kbb^\star q}{= 0}{}{\label{eq:bilevel_mpec_formulationTV constraint_1}}
    \addConstraint{(\Kbb u)_i}{= r_i [\cos(\theta_i),\sin(\theta_i)]^\intercal,}{\quad i=1,\dots,m}
    \addConstraint{q_i}{= \delta_i [\cos(\theta_i),\sin(\theta_i)]^\intercal,}{\quad i=1,\dots,m}{\label{eq:bilevel_mpec_formulationTV constraint_3}}
    \addConstraint{\delta_i}{\ge 0,}{\quad i=1\dots,m}
    \addConstraint{0}{\le r_i \perp (Q(\alpha)_i-\delta_i) \ge 0,}{\quad i=1\dots,m,}{\label{eq:bilevel_mpec_formulationTV constraint_6}}
\end{mini!}
and let us introduce the active, inactive and biactive sets:
\begin{align*}
  \Acal (x)&{}:=\{i:r_{i} =0, Q(\alpha)_i> \delta_{i} \},\\
  \Ical (x)&{}:=\{i:r_{i} >0, Q(\alpha)_i= \delta_{i}\},\\
  \Bcal (x)&{}:=\{i:r_{i} =0, Q(\alpha)_i= \delta_{i}\}.
\end{align*}


\begin{theorem} \label{thm: strong stationarity MPCC}
 Let $x^*=(u^*,q^*,\alpha^*,r^*,\delta^*,\theta^*)$ be a local optimal solution to \eqref{eq:bilevel_mpec_formulationTV}. Assume
 that the set $\mathcal T(x^*):=\{i: r_i^*= \delta^*_i = 0 \}$ is empty, and that if $\xi \neq 0$ is such that $\nabla_\alpha Q(\alpha^*)_{\mathcal Z} ~\xi  =0$, where $\mathcal Z := \{ i: \delta_i^* = Q(\alpha_i^*)=0 \}$,
 then $\xi$ has both positive and negative components. Moreover, assume that $\cos(\theta^*_i) \neq 0, ~\forall i \in \Acal(x^*) \cup \Bcal(x^*),$ and that $\left[ (K_y)_{\Acal(x^*) \cup \Bcal(x^*)} - \diag(\tan(\theta_{\Acal(x^*) \cup \Bcal(x^*)})) (K_x)_{\Acal(x^*) \cup \Bcal(x^*)} \right]^\intercal$ is full-rank, for any $x$ in a neighbourhood of $x^*$.
 Then there exist Lagrange multipliers $p \in \R^d$ and
  $\lambda^x, \lambda^y, \varphi^x, \varphi^y, \rho, \sigma, \gamma, \nu \in \R^m$	such that, together with equations \eqref{eq:bilevel_mpec_formulationTV constraint_1}-\eqref{eq:bilevel_mpec_formulationTV constraint_6}, the following \textbf{M-stationarity} system is satisfied:
\begin{subequations} \label{eq:strong_stationarity_TV}
\begin{align}
    \nabla_u \Jcal(u^*) - \nabla_{uu}^2 \Dcal_0 (u^*)^\intercal p - \Kbb^\intercal \varphi & = 0, \label{eq:strong_stationarity_TV_1}\\
    \nabla_\alpha Q(\alpha^*) ~\nu & =0 \label{eq:strong_stationarity_TV_2}\\
    - K_x p  + \lambda^x &=0 \label{eq:strong_stationarity_TV_3}\\
    - K_y p + \lambda^y &=0 \label{eq:strong_stationarity_TV_4}\\
    \cos(\theta^*) \circ \varphi^x + \sin(\theta^*) \circ \varphi^y - \gamma & =0 \label{eq:strong_stationarity_TV_5}\\
    -\cos(\theta^*) \circ \lambda^x -\sin(\theta^*) \circ \lambda^y - \sigma + \nu & =0 \label{eq:strong_stationarity_TV_6}\\
    \sin(\theta^*) \circ \left( -r^* \circ \varphi^x + \delta^* \circ \lambda^x \right) -\cos(\theta^*) \circ \left( - r^* \circ \varphi^y + \delta^* \circ \lambda^y \right)& =0 \label{eq:strong_stationarity_TV_7}\\
    0 \le \delta^* \perp \sigma & \ge 0, \label{eq:strong_stationarity_TV_9}\\
    \nu_i &=0,\; &&\forall i\in \Acal(x^*), \label{eq:strong_stationarity_TV_10}\\
    \gamma_i &= 0, \;&&\forall i\in \Ical(x^*), \label{eq:strong_stationarity_TV_11}\\
    \gamma_i  \nu_i=0 \lor \gamma_i \ge 0,~ \nu_i &\ge 0,\;&&\forall i\in \Bcal(x^*), \label{eq:strong_stationarity_TV_12}
\end{align}
\end{subequations}
where $\circ$ stands for the Hadamard product.
If, in addition,
\begin{equation}
 \nabla_\alpha Q(\alpha^*)_{\{\delta^*_i = Q(\alpha^*)_i \}} \zeta =0 \implies  \zeta_i =0, ~\forall i \in \Bcal (x^*), \label{eq: s stationarity extra hypothesis}
\end{equation}
and $q|_{\Acal^* \cup \Bcal^*} \in \operatorname{range} (\Kbb_{\Acal^* \cup \Bcal^*})$,
then $x^*$ is \textbf{S-stationary} and \Cref{eq:strong_stationarity_TV_12} is replaced by
\begin{equation} \label{eq:strong_stationarity_TV_Strong}
  \gamma_i \ge 0,~ \nu_i \ge 0, \quad \forall i\in \Bcal(x^*).
\end{equation}
\end{theorem}
\begin{proof}
We start by writing \eqref{eq:bilevel_mpec_formulationTV} in the form of \eqref{eq:MPCC}. To do so, we introduce the vector $x = (u,q,\alpha,r,\delta,\theta)^\intercal \in \R^{n}$, with $n=d+|\alpha|+ 5m$, and define $h: \R^{n} \to \R^{d+4m}$ by
\begin{equation*}
  h(x)=
  \begin{bmatrix}
    & - \nabla_u \Dcal_0 (u) - \Kbb^\star q\\
    & -K_x u + r \circ \cos (\theta)\\
    & -K_y u + r \circ \sin (\theta)\\
    & q^x - \delta \circ \cos (\theta)\\
    & q^y - \delta \circ \sin (\theta)
  \end{bmatrix},
\end{equation*}
where we used the structure of the TV problem to rewrite the constraints in terms of the discrete partial derivative operators. In addition, we define $g: \R^{n} \to \R^{m}$ by $g(x)=-\delta$
and $M: \R^{n} \to \R^{m}$, $N: \R^{n} \to \R^{m}$ by $M(x)= r$ and $N(x)= Q(\alpha) -\delta$, respectively.

The gradients of the constraints are then given by
\begin{equation*}
  \nabla h(x)=
  \begin{bsmallmatrix}
    & -\nabla^2_{u u} \Dcal_0 (u)^\intercal & -K_x^\intercal & -K_y^\intercal &\mathbf{0} &\mathbf{0}\\
    & -K_x    &\mathbf{0} &\mathbf{0} & I &\mathbf{0}\\
    & -K_y    &\mathbf{0} &\mathbf{0} &\mathbf{0} & I\\
    & \mathbf{0}  &\mathbf{0} &\mathbf{0} &\mathbf{0} &\mathbf{0}\\
    & \mathbf{0} & \Cos (\theta) & \Sin (\theta) &\mathbf{0} &\mathbf{0}\\
    & \mathbf{0} & \mathbf{0} & \mathbf{0} & -\Cos (\theta) & -\Sin (\theta)\\
    & \mathbf{0} & -R \Sin(\theta) & R \Cos(\theta) & D \Sin(\theta)) &-D  \Cos(\theta))
  \end{bsmallmatrix},
\end{equation*}
  \begin{equation*}
    \nabla g(x)=
    \begin{bsmallmatrix}
      & \mathbf{0}\\
      & \mathbf{0}\\
      & \mathbf{0}\\
      & \mathbf{0}\\
      & \mathbf{0}\\
      &-I\\
      &\mathbf{0}
    \end{bsmallmatrix}, \qquad
  \nabla M(x)=
  \begin{bsmallmatrix}
    & \mathbf{0} \\
    & \mathbf{0} \\
    & \mathbf{0} \\
    & \mathbf{0}\\
    & I\\
    & \mathbf{0}\\
    & \mathbf{0}
  \end{bsmallmatrix}, \qquad
  \nabla N(x)=
  \begin{bsmallmatrix}
    & \mathbf{0} \\
    & \mathbf{0} \\
    & \mathbf{0} \\
    & \nabla_\alpha Q(\alpha)\\
    & \mathbf{0}\\
    & -I\\
    & \mathbf{0}
  \end{bsmallmatrix},
\end{equation*}
where we used the diagonal matrix notation $\Cos(\theta) := \diag(\cos(\theta))$, $\Sin(\theta) := \diag(\sin(\theta))$, $R := \diag(r)$ and $D := \diag(\delta)$.

To verify MPCC-RCPLD, let us start by analyzing the rank of the equality constraints matrix $\nabla h(x)$, with $x$ sufficiently close to $x^*$. The linear system $\sum_{i=1}^{d+4m} \psi_i \nabla h_i(x)=0$ can be reformulated, using $\psi = (v,\varphi,\pi)$, as
\begin{subequations} \label{eq: linear system in h CQ for TV 0}
  \begin{align}
    & \nabla^2_{u u} \Dcal_0 (u)^\intercal v + K_x^\intercal \varphi^x + K_y^\intercal \varphi^y = 0 \label{eq: linear system in h CQ for TV 1}\\
    & K_x v - \pi^x =0 \label{eq: linear system in h CQ for TV 2}\\
    & K_y v - \pi^y =0 \label{eq: linear system in h CQ for TV 3}\\
    & \Cos(\theta) \varphi^x + \Sin(\theta) \varphi^y = 0 \label{eq: linear system in h CQ for TV 4}\\
    & -\Cos(\theta) \pi^x - \Sin(\theta) \pi^y = 0 \label{eq: linear system in h CQ for TV 5}\\
    & - R \Sin(\theta) \varphi^x + R \Cos(\theta) \varphi^y + D \Sin(\theta) \pi^x - D \Cos(\theta) \pi^y =0, \label{eq: linear system in h CQ for TV 6}
  \end{align}
\end{subequations}

From \Cref{eq: linear system in h CQ for TV 2} and \Cref{eq: linear system in h CQ for TV 3} it immediately follows that $\pi = \Kbb v$. Moreover, from equations \eqref{eq: linear system in h CQ for TV 5}-\eqref{eq: linear system in h CQ for TV 6}, and since $\mathcal T(x) \subset \mathcal T(x^*) = \emptyset$, for all $x$ in a neighbourhood of $x^*$, we also get that
\begin{equation*}
  \begin{pmatrix}
    -\cos (\theta_i) & -\sin (\theta_i)\\
     \sin (\theta_i) &  -\cos (\theta_i)
  \end{pmatrix}
  \begin{pmatrix}
    \pi^x_i\\
    \pi^y_i
  \end{pmatrix}
  =
  \begin{pmatrix}
    0\\0
  \end{pmatrix}, \quad \forall i \in \Acal(x) \cup \Bcal(x),
\end{equation*}
which, thanks to the orthogonality of the matrix, implies that $\pi_i=(\Kbb v)_i=0, ~ \forall i \in \Acal(x) \cup \Bcal(x).$

Combining \Cref{eq: linear system in h CQ for TV 4} and \Cref{eq: linear system in h CQ for TV 6} we obtain, for $i \in \Ical(x)$, that
\begin{equation*}
  \begin{pmatrix}
    \cos (\theta_i) & \sin (\theta_i)\\
     -r_i \sin (\theta_i) &  r_i \cos (\theta_i)
  \end{pmatrix}
  \begin{pmatrix}
    \varphi^x_i\\
    \varphi^y_i
  \end{pmatrix}
  =
  \begin{pmatrix}
    0\\-\delta_i \left(\sin(\theta_i) \pi^x_i -\cos(\theta_i)\pi^y_i \right)
  \end{pmatrix},
\end{equation*}
Since $r_i >0,~\forall i \in \Ical(x)$, the matrix on the left hand side becomes orthogonal and we obtain
\begin{align*}
  \begin{pmatrix}
    \varphi^x_i\\
    \varphi^y_i
  \end{pmatrix}
  & =
  \begin{pmatrix}
    \cos (\theta_i) & -\sin (\theta_i)\\
     \sin (\theta_i) &  \cos (\theta_i)
  \end{pmatrix}
  \begin{pmatrix}
    0\\-\delta_i r_i^{-1} \left(\sin(\theta_i) (K_x v)_i -\cos(\theta_i)(K_y v)_i \right)
  \end{pmatrix}\\
  &=Q(\alpha)_i r_i^{-1}
  \begin{pmatrix}
    \sin(\theta_i)^2 (K_x v)_i- \sin(\theta_i) \cos(\theta_i) (K_y v)_i\\
    - \sin(\theta_i) \cos(\theta_i) (K_x v)_i +\cos(\theta_i)^2(K_y v)_i
  \end{pmatrix},
\end{align*}
which implies that
\begin{equation} \label{eq: RCPLD value of multiplier for grad h on I}
  \varphi_i = \frac{Q(\alpha)_i}{\|(\Kbb u)_i\|} \left( (\Kbb v)_i- \frac{\scalar{(\Kbb u)_i}{(\Kbb v)_i}}{\|(\Kbb u^*)_i\|^2} (\Kbb u)_i \right), \quad \forall i \in \Ical(x).
\end{equation}
Multiplying the latter with $(\Kbb v)_i$ we then obtain that $\scalar{\varphi_i}{(\Kbb v)_i} \geq 0$, for all $i \in \Ical(x)$.

Altogether, we arrive at the reduced system
\begin{subequations} \label{eq: reduced system grad h}
  \begin{align}
    & \nabla^2_{u u} \Dcal_0 (u)^\intercal v + \Kbb^\star \varphi = 0,\\
    & \varphi_i = \frac{Q(\alpha)_i}{|(\Kbb u)_i|} \left( (\Kbb v)_i- \frac{\scalar{(\Kbb u)_i}{(\Kbb v)_i}}{|(\Kbb u^*)_i|^2} (\Kbb u)_i \right), &&\forall i \in  \Ical(x),\\
    & (\Kbb v)_i =0, &&\forall i \in  \Acal(x) \cup \Bcal(x).
  \end{align}
\end{subequations}
Thanks to the continuity and positive definiteness of $\nabla^2_{u u} \Dcal_0 (u^*)$, and since $\scalar{\varphi_i}{(\Kbb v)_i} \geq 0$, for all $i \in \Ical(x)$, the vector $v=\mathbf{0}$ is the unique solution to the reduced system \eqref{eq: reduced system grad h}, which directly implies that $\varphi_i=0$, for all $i \in \Ical(x)$, and also $\Kbb^\star \varphi = 0$. Together with \Cref{eq: linear system in h CQ for TV 4}, we then obtain 
\begin{equation} \label{eq: reduced linear system RCPLD rank of grad h}  
  \begin{pmatrix}
    K_x^\intercal &K_y^\intercal\\
    \chi_{\Ical(x)}+ \chi_{\Acal(x) \cup \Bcal(x)} \Cos(\theta) &\chi_{\Acal(x) \cup \Bcal(x)} \Sin(\theta)\\
    0 & \chi_{\Ical(x)}
  \end{pmatrix}
  \begin{pmatrix}
    \varphi^x\\
    \varphi^y
  \end{pmatrix}
  =
  \begin{pmatrix}
    0\\
    0\\
    0
  \end{pmatrix},
\end{equation}
where $\chi_E$ denotes the diagonal matrix with the indicator function of a set $E$ on the main diagonal. Since $\Acal(x) \cup \Bcal(x) \subset \Acal(x^*) \cup \Bcal(x^*)$ and, by hypotheses, $\cos(\theta^*_i) \neq 0$, for all $i \in \Acal(x^*) \cup \Bcal(x^*)$, there exists a neighbourhood $V(x^*)$ such that, for any $x \in V(x^*)$, $\cos(\theta_i) \neq 0$, $\forall i \in \Acal(x) \cup \Bcal(x).$ Moreover, since $\left[ (K_y)_{\Acal(x^*) \cup \Bcal(x^*)} - \Tan(\theta_{\Acal(x^*) \cup \Bcal(x^*)}) (K_x)_{\Acal(x^*) \cup \Bcal(x^*)} \right]^\intercal$ is full-rank, for any $x \in V(x^*)$,
it can be easily verified that $\varphi =0$ is the unique solution to \eqref{eq: reduced linear system RCPLD rank of grad h}. Therefore, the matrix $\nabla h(x)$ is full-rank for any $x \in V(x^*)$, and necessarily $I_1 = \{1,\dots,p\}$.

Let now $I_2 \subset \{i: g_i(x^*)=0 \}$, $I_3 \subset \Acal(x^*) \cup \Bcal(x^*)$ and $I_4 \subset \Ical(x^*) \cup \Bcal(x^*)$ be arbitrary but fix, and let us analyze the solutions $(\psi, \mu, \gamma, \nu)$ to
\begin{subequations}\label{eq: positive linear system}
  \begin{align}
    & \sum_{i=1}^{d+4m} \psi_i \nabla h_i(x^*) + \sum_{I_2} \mu_i \nabla g_i(x^*) - \sum_{I_3} \gamma_i \nabla M_i(x^*) - \sum_{I_4} \nu_i \nabla N_i(x^*) = 0, \label{eq: positive linear system 1}\\
    & \mu_i \geq 0, \quad \forall i \in I_2,\label{eq: positive linear system 2}\\
    & \text{either } \gamma_i>0, ~ \nu_i>0 \quad \text{or} \quad \nu_i \gamma_i=0, \quad \forall i \in \Bcal(x^*).\label{eq: positive linear system 3}
\end{align}
\end{subequations}
\Cref{eq: positive linear system 1} can also be written in more detail, with $\psi = (v,\varphi,\pi)$, as
\begin{subequations} \label{eq: linear system CQ for TV 0}
  \begin{align}
    & \nabla^2_{u u} \Dcal_0 (u^*)^\intercal v + K_x^\intercal \varphi^x + K_y^\intercal \varphi^y = 0 \label{eq: linear system CQ for TV 1}\\
    & K_x v - \pi^x =0 \label{eq: linear system CQ for TV 2}\\
    & K_y v - \pi^y =0 \label{eq: linear system CQ for TV 3}\\
    & \sum_{I_4 \subset \{\delta^*_i=Q(\alpha^*)_i\}}\nu_i \nabla_{\alpha}Q_i(\alpha^*) =0, \label{eq: linear system CQ for TV 3.5}\\
    & \Cos(\theta^*) \varphi^x + \Sin(\theta^*) \varphi^y = \chi_{I_3} \gamma \label{eq: linear system CQ for TV 4}\\
    & -\Cos(\theta^*) \pi^x - \Sin(\theta^*) \pi^y - \chi_{I_2}\mu+ \chi_{I_4} \nu= 0 \label{eq: linear system CQ for TV 5}\\
    & - R^* \Sin(\theta^*) \varphi^x + R^* \Cos(\theta^*) \varphi^y + D^* \Sin(\theta^*) \pi^x - D^* \Cos(\theta^*) \pi^y =0. \label{eq: linear system CQ for TV 6}
  \end{align}
\end{subequations}

If $(v,\varphi,\pi , \mu, \nu,\gamma)$ solves \eqref{eq: linear system CQ for TV 0}, then, from \Cref{eq: linear system CQ for TV 2} and \Cref{eq: linear system CQ for TV 3}, it follows that $\pi = \Kbb v$. Moreover, from equations \eqref{eq: linear system CQ for TV 5}-\eqref{eq: linear system CQ for TV 6}, and since $\mathcal T(x^*) = \emptyset$, it follows that
\begin{equation*}
  \begin{pmatrix}
    -\cos (\theta_i^*) & -\sin (\theta_i^*)\\
     \sin (\theta_i^*) &  -\cos (\theta_i^*)
  \end{pmatrix}
  \begin{pmatrix}
    \pi^x_i\\
    \pi^y_i
  \end{pmatrix}
  =
  \begin{pmatrix}
    0\\0
  \end{pmatrix}, \quad \forall i \in \Acal(x^*),
\end{equation*}
which implies that $\pi_i=(\Kbb v)_i=0, ~ \forall i \in \Acal(x^*).$
From \Cref{eq: linear system CQ for TV 4} we also get that
\begin{equation} \label{eq: proof CQ product lambda and q}
  \scalar{\varphi_i}{[\cos(\theta^*_i) ~\sin(\theta^*_i)]^\intercal} =
  \begin{cases}
    \gamma_i, & \forall i \in I_3 \subset \Acal(x^*) \cup \Bcal(x^*)\\
    0, & \text{else}.
  \end{cases}
\end{equation}

From \Cref{eq: linear system CQ for TV 4} and \Cref{eq: linear system CQ for TV 6} we get, for $i \in \Ical(x^*)$, that
\begin{equation*}
  \begin{pmatrix}
    \cos (\theta_i^*) & \sin (\theta_i^*)\\
     -r_i^* \sin (\theta_i^*) &  r_i^* \cos (\theta_i^*)
  \end{pmatrix}
  \begin{pmatrix}
    \varphi^x_i\\
    \varphi^y_i
  \end{pmatrix}
  =
  \begin{pmatrix}
    0\\-\delta_i^* \left(\sin(\theta_i^*) \pi^x_i -\cos(\theta_i^*)\pi^y_i \right)
  \end{pmatrix}.
\end{equation*}
Since $r_i^* >0,~\forall i \in \Ical(x^*)$, we obtain, similarly as in the derivation of \Cref{eq: RCPLD value of multiplier for grad h on I}, that
\begin{equation}
  \varphi_i = \frac{Q(\alpha^*)_i}{\|(\Kbb u^*)_i\|} \left( (\Kbb v)_i- \frac{\scalar{(\Kbb u^*)_i}{(\Kbb v)_i}}{\|(\Kbb u^*)_i\|^2} (\Kbb u^*)_i \right), \quad \forall i \in \Ical(x^*).
\end{equation}
Multiplying the latter with $(\Kbb v)_i$ we also get that $\scalar{\varphi_i}{(\Kbb v)_i} \geq 0$, for all $i \in \Ical(x^*)$.

Let us now introduce the following partition of the biactive set:
\begin{align*}
  & \Bcal_{0+}:= \{i \in \Bcal(x^*): \nu_i =0 \}, \\
  & \Bcal_{+0}:= \{i \in \Bcal(x^*): \gamma_i=0,~\nu_i \neq 0 \}, \\
  & \Bcal_{++}:= \{i \in \Bcal(x^*): \nu_i >0, \gamma_i>0 \}.
\end{align*}
On the index set $\Bcal_{0+}$, similarly as for $\Acal(x^*)$, it directly follows that
\begin{equation}
  (\Kbb v)_i =0, \quad \forall i \in \Bcal_{0+}.
\end{equation}
On the other hand, for an index $i \in \Bcal_{+0}$, we get from \Cref{eq: proof CQ product lambda and q} that $\scalar{\varphi_i}{[\cos(\theta^*_i) ~\sin(\theta^*_i)]^\intercal} = 0$.
Additionally, from \Cref{eq: linear system CQ for TV 5} and \Cref{eq: linear system CQ for TV 6} we also obtain that
\begin{equation*}
  \begin{pmatrix}
    \pi^x_i\\
    \pi^y_i
  \end{pmatrix}
  =
  \begin{pmatrix}
    \cos (\theta_i^*) & -\sin (\theta_i^*)\\
     \sin (\theta_i^*) &  \cos (\theta_i^*)
  \end{pmatrix}
  \begin{pmatrix}
    \nu_i\\0
  \end{pmatrix} = \nu_i [\cos(\theta^*_i) ~\sin(\theta^*_i)]^\intercal,
\end{equation*}
which implies that $(\Kbb v)_i = \nu_i [\cos(\theta^*_i) ~\sin(\theta^*_i)]^\intercal$, $|\nu_i| = \| (\Kbb v)_i \|$ and, consequently,
\begin{equation}
  \scalar{\varphi_i}{(\Kbb v)_i} = 0, \quad \forall i \in \Bcal_{+0}.
\end{equation}
On the index set $\Bcal_{++}$ we obtain from \Cref{eq: linear system CQ for TV 5} that
\begin{equation} \label{eq: proof CQ B++ ineq}
  0 < \nu_i = \scalar{(\Kbb v)_i}{[\cos(\theta^*_i) ~\sin(\theta^*_i)]^\intercal}.
\end{equation}
Additionally, from \Cref{eq: linear system CQ for TV 5} and \Cref{eq: linear system CQ for TV 6} we also get that
$(\Kbb v)_i = \nu_i [\cos(\theta^*_i) ~\sin(\theta^*_i)]^\intercal$ and, together with \Cref{eq: proof CQ B++ ineq}, it follows that $\nu_i = \|(\Kbb v)_i\| > 0.$
Jointly with \Cref{eq: linear system CQ for TV 4} we then obtain that
\begin{equation} \label{eq: proof CQ B++ ineq 2}
  0 < \gamma_i = \scalar{\varphi_i}{[\cos(\theta^*_i) ~\sin(\theta^*_i)]^\intercal}= \nu_i^{-1} \scalar{\varphi_i}{(\Kbb v)_i}.
\end{equation}
Multiplying \Cref{eq: linear system CQ for TV 1} with $v$, using the positive definiteness of $\nabla^2_{u u} \Dcal_0 (u^*)$, and replacing all obtained values of $\Kbb v$ and $\varphi$, it then follows that
$\sum_{\Bcal_{++}} \scalar{\varphi_i}{(\Kbb v)_i} \leq 0,$ in contradiction to \Cref{eq: proof CQ B++ ineq 2}. Therefore, $\Bcal_{++} = \emptyset$.

Consequently, we get the reduced system
\begin{subequations} \label{eq: reduced lin. system}
  \begin{align}
    & \nabla^2_{u u} \Dcal_0 (u^*)^\intercal v + \Kbb^\star \varphi = 0, \label{eq: reduced lin. system 1}\\
    & \varphi_i = \frac{Q(\alpha^*)_i}{|(\Kbb u^*)_i|} \left( (\Kbb v)_i- \frac{\scalar{(\Kbb u^*)_i}{(\Kbb v)_i}}{|(\Kbb u^*)_i|^2} (\Kbb u^*)_i \right), &&\forall i \in  \Ical(x^*),  \label{eq: reduced lin. system b}\\
    & (\Kbb v)_i =0, &&\forall i \in  \Acal(x^*) \cup \Bcal_{0+}, \\
    & \scalar{\varphi_i}{(\Kbb v)_i} =0, &&\forall i \in  \Bcal_{+0},
  \end{align}
\end{subequations}
which, thanks to the positive definiteness of $\nabla^2_{u u} \Dcal_0 (u^*)$ and the fact that $\scalar{\varphi_i}{(\Kbb v)_i} \geq 0, ~\forall i$, has $v=0$ as its unique solution.
From \eqref{eq: reduced lin. system b} we directly get that 
\begin{equation} \label{eq: CQ TV varphi=0}
  \varphi_i=0, \qquad \forall i \in \Ical(x^*),
\end{equation}
and, from \Cref{eq: reduced lin. system 1}, $\Kbb^\star \varphi=0$. Notice that the value of $\varphi_i$ on $\Acal(x^*) \cup \Bcal(x^*)$ is not uniquely determined as $\ker (\Kbb^\star) \neq \{ \mathbf{0}\}$. Additionally, on the biactive set we get, since $\nu_i = \|(\Kbb v)_i\|$, for all $i \in \Bcal_{+0}$, that $\nu_i = 0,$ for all $i \in \Bcal(x^*)$. \Cref{eq: linear system CQ for TV 5} then yields
\begin{equation*}
  \nu_i =
  \begin{cases}
    0 &\text{ if }i \notin \{\delta_i=0\},\\
    \mu_i &\text{ if }i \in \{\delta_i=0\}.
  \end{cases}
\end{equation*}
From \Cref{eq: linear system CQ for TV 3.5} it follows that $\sum_{\{\delta_i=Q(\alpha^*)_i=0\}} \nu_i \nabla_{\alpha}Q_i(\alpha^*) =0$ which, thanks to the hypothesis of the theorem and the required nonnegativity of $\mu$, implies that $\mu_i = \nu_i =0$, for all $i \in \{\delta^*_i=0\}.$
Consequently, from Equations \eqref{eq: linear system CQ for TV 1}, \eqref{eq: linear system CQ for TV 4} and \eqref{eq: CQ TV varphi=0}, the multiplier $\varphi$ has to solve the system
\[
  \begin{pmatrix}
    K_x^\intercal &K_y^\intercal\\
    \chi_{\Ical^*}+ \chi_{\Acal^* \cup \Bcal^*} \Cos(\theta^*) &\chi_{\Acal^* \cup \Bcal^*} \Sin(\theta^*)\\
    0 & \chi_{\Ical^*}
  \end{pmatrix}
  \begin{pmatrix}
    \varphi^x\\
    \varphi^y
  \end{pmatrix}
  =
  \begin{pmatrix}
    0\\
    \chi_{I_3} \gamma\\
    0
  \end{pmatrix},
\]
where $\Ical^* := \Ical(x^*)$, $\Acal^* := \Acal(x^*)$, $\Bcal^* := \Bcal(x^*)$. Thanks to the full-rank hypothesis on the matrix $\left[ (K_y)_{\Acal^* \cup \Bcal^*} - \Tan(\theta^*_{\Acal^* \cup \Bcal^*}) (K_x)_{\Acal^* \cup \Bcal^*} \right]^\intercal$, we obtain that for $\chi_{I_3} \gamma \neq 0$, there is a unique solution $\varphi \neq 0$. Hence, there exists a nonzero multiplier $(\psi,\mu,\gamma,\nu)$ solution of system \eqref{eq: linear system CQ for TV 0}.

Let us now consider, for $x \in V(x^*)$, the system
\begin{equation*}
  \sum_{i=1}^{d+4m} \tilde \psi_i \nabla h_i(x) + \sum_{I_2} \tilde \mu_i \nabla g_i(x) - \sum_{I_3} \tilde \gamma_i \nabla M_i(x) - \sum_{I_4} \tilde \nu_i \nabla N_i(x) = 0,
\end{equation*}
or, equivalently,
\begin{subequations}
  \begin{align}
    & \nabla^2_{u u} \Dcal_0 (u)^\intercal \tilde v + K_x^\intercal \tilde \varphi^x + K_y^\intercal \tilde \varphi^y = 0\\
    & K_x \tilde v - \tilde \pi^x =0\\
    & K_y \tilde v - \tilde \pi^y =0\\
    & \sum_{I_4}\tilde \nu_i \nabla_{\alpha}Q_i(\alpha) =0,\\
    & \Cos(\theta) \tilde \varphi^x + \Sin(\theta) \tilde \varphi^y = \chi_{I_3} \tilde \gamma \\
    & -\Cos(\theta) \tilde \pi^x - \Sin(\theta) \tilde \pi^y - \chi_{I_2}\tilde \mu+ \chi_{I_4} \tilde \nu= 0 \\
    & - R \Sin(\theta) \tilde \varphi^x + R \Cos(\theta) \tilde \varphi^y + D \Sin(\theta) \tilde \pi^x - D \Cos(\theta) \tilde \pi^y =0.
  \end{align}
\end{subequations}
 Taking $\tilde v=0$, $\tilde \pi=0$, $\tilde \nu=0$ and $\tilde \mu=0$, we arrive at the reduced system 
\[
  \begin{pmatrix}
    K_x^\intercal &K_y^\intercal\\
    \chi_{\Ical(x)}+ \chi_{\Acal(x) \cup \Bcal(x)} \Cos(\theta) &\chi_{\Acal(x) \cup \Bcal(x)} \Sin(\theta)\\
    0 & \chi_{\Ical(x)}
  \end{pmatrix}
  \begin{pmatrix}
    \varphi^x\\
    \varphi^y
  \end{pmatrix}
  =
  \begin{pmatrix}
    0\\
    \chi_{I_3} \tilde \gamma\\
    0
  \end{pmatrix},
\]
which, thanks again to the full-rank hypothesis on the matrix $\left[ (K_y)_{\Acal^* \cup \Bcal^*} - \Tan(\theta_{\Acal^* \cup \Bcal^*}) (K_x)_{\Acal^* \cup \Bcal^*} \right]^\intercal$, implies that for any $\chi_{I_3} \tilde \gamma \neq 0$, there is a unique solution $\tilde \varphi \neq 0$. 

Altogether, we have proved that MPCC-RCPLD is fulfilled, which implies MPCC-ACQ and the M-stationarity conditions.

If, in addition, \Cref{eq: s stationarity extra hypothesis} holds, then, from \Cref{eq: linear system CQ for TV 3.5}, with $I_4 = \Ical(x^*) \cup \Bcal(x^*)$, we get that $\nu_i=0$ for all $i \in \Bcal^*$. Proceeding in a similar manner as in the derivation of system \eqref{eq: reduced lin. system}, with $\Bcal^*= \Bcal_{0+}$, we get that the solution to system
\eqref{eq: linear system CQ for TV 1}-\eqref{eq: linear system CQ for TV 6}, with $I_2=\{ g_i(x^*)=0 \}$, $I_3 = \Acal(x^*) \cup \Bcal(x^*)$ and $I_4 = \Ical(x^*) \cup \Bcal(x^*)$, satisfies $v=0$, $\Kbb^\intercal \varphi=0$ and $\varphi_i =0,~\forall i \in \Ical^*$.
If also
$q|_{\Acal^* \cup \Bcal^*} \in \operatorname{range}
\begin{pmatrix}
  (K_x)_{\Acal^* \cup \Bcal^*}\\
  (K_y)_{\Acal^* \cup \Bcal^*}
\end{pmatrix}$,
then, thanks to \Cref{eq: linear system CQ for TV 4},
\[
  \begin{aligned}
    \gamma_{\Acal^* \cup \Bcal^*} &= \varphi^x_{\Acal^* \cup \Bcal^*} \circ \cos(\theta_{\Acal^* \cup \Bcal^*}^*)+ \varphi^y_{\Acal^* \cup \Bcal^*} \circ \sin(\theta^*_{\Acal^* \cup \Bcal^*}) = \delta_{\Acal^* \cup \Bcal^*}^{-1} \circ \scalar{q^*_{\Acal^* \cup \Bcal^*}}{\varphi_{\Acal^* \cup \Bcal^*}}\\
    &= \delta_{\Acal^* \cup \Bcal^*}^{-1} \circ \scalar{\varphi_{\Acal^* \cup \Bcal^*}}{\Kbb_{\Acal^* \cup \Bcal^*} \tilde w}  =  \delta_{\Acal^* \cup \Bcal^*}^{-1} \circ \scalar{(\Kbb_{\Acal^* \cup \Bcal^*})^\intercal ~\varphi_{\Acal^* \cup \Bcal^*}}{\tilde w} =0,
  \end{aligned}
\]
where $\tilde w$ is such that $\Kbb_{\Acal^* \cup \Bcal^*} \tilde w = q^*_{\Acal^* \cup \Bcal^*}$. Consequently, we get that any solution to equations \eqref{eq: positive linear system 1}-\eqref{eq: positive linear system 2} satisfies  $\gamma_i=\nu_i=0$, for all $i \in \Bcal(x^*)$,
and, therefore, partial MPCC-LICQ holds. By applying Theorem \ref{them: strong stationarity} the S-stationarity system follows.
\end{proof}

System \eqref{eq:strong_stationarity_TV} may also be written using solely the original variables, leading to stationary systems sharper than the ones obtained previously.
\begin{theorem} \label{thm:strong_stationarity_TV_original_variables}
  Let $x^*=(u^*,q^*,\alpha^*)$ be an optimal solution to \eqref{eq:bilevel_problem_kkt_reformulation} and, for all $i=1,\dots,m,$ let $r_i := \|(\Kbb u)_i\|$, $\delta_i := \|q_i\|$ and $\theta^*_i$ be the associated angles according to the reformulation \eqref{eq:bilevel_mpec_formulation}.
  Assume that the same hypotheses of \Cref{thm: strong stationarity MPCC} hold.
  Then there exist Lagrange multipliers
  $p \in \R^d, \varphi \in \R^{2m}$, and $\sigma, \nu \in \R^m$ such that the following \textbf{M-stationarity} system is satisfied:
  \begin{subequations} \label{eq:strong_stationarity_primal_variables}
  \begin{align}
      & \nabla_u \Dcal_0(u^*) + \Kbb^\star q^* = 0 \label{eq:strong_stationarity_primal_1}\\
      & \scalar{q^*_i}{(\Kbb u^*)_i} = Q(\alpha^*)_i  \|(\Kbb u^*)_i\|,\; && \forall i=1,\dots,m \label{eq:strong_stationarity_primal_2}\\
      & \|q^*_i\| \le Q(\alpha^*)_i,\; && \forall i=1,\dots,m \label{eq:strong_stationarity_primal_3}\\
       & \nabla_{uu}^2 \Dcal_0 (u^*)^\intercal p + \Kbb^\star \varphi = \nabla_u \Jcal (u^*), \label{eq:strong_stationarity_primal_4}\\
      & \nabla_\alpha Q(\alpha^*) ~ \nu =0 \label{eq:strong_stationarity_primal_5}\\
      & 0 \le \|q_i\| \perp \sigma_i \ge 0\; && \forall i=1,\dots,m \label{eq:strong_stationarity_primal_6.5}\\
      & \|(\Kbb u^*)_i\| ~\varphi_i = Q(\alpha^*)_i \left( I- \frac{(\Kbb u)_i (\Kbb u)_i^\intercal}{\|(\Kbb u^*)_i\|^2} \right) (\Kbb  p)_i, \; && \forall i \in \Ical(x^*) \label{eq:strong_stationarity_primal_7}\\
      & \left( \Kbb p \right)_i =0, \; && \forall i \in \Acal(x^*)  \label{eq:strong_stationarity_primal_8}\\[10pt]
      & \left.
        \begin{aligned}
          & (\Kbb p)_i =0, \quad \lor \\
          & (\Kbb p)_i \in \spa{(q^*_i)}, ~\scalar{\varphi_i}{q^*_i} = 0, \quad \lor \\
          & \scalar{q^*_i}{(\Kbb p)_i} = Q(\alpha^*)_i \|( \Kbb p )_i\|, ~\scalar{\varphi_i}{q^*_i} \ge 0.
       \end{aligned}
     \right\} && \forall i\in \Bcal(x^*) \label{eq:strong_stationarity_primal_9}\\[10pt]
      & \nu_i = \begin{dcases} \frac{1}{\|(\Kbb u)_i\|} \scalar{(\Kbb u)_i}{(\Kbb p)_i}+ \sigma_i, \; & \forall i \in \Ical(x^*)\\
      0, \; & \forall i \in \Acal(x^*)\\
      \frac{1}{Q(\alpha)_i}\scalar{ q_i}{(\Kbb p)_i}, \; & \forall i \in \Bcal(x^*)
    \end{dcases} \label{eq:strong_stationarity_primal_11}
  \end{align}
  \end{subequations}
  If, in addition, condition \eqref{eq: s stationarity extra hypothesis} holds and $q|_{\Acal^* \cup \Bcal^*} \in \operatorname{range} (\Kbb_{\Acal^* \cup \Bcal^*})$,
  then $x^*$ is \textbf{S-stationary} and \Cref{eq:strong_stationarity_primal_9} is replaced by
  \begin{equation} \label{eq:strong_stationarity_primal_12}
    \scalar{q^*_i}{(\Kbb p)_i} = Q(\alpha^*)_i \|( \Kbb p )_i\| \quad \text{ and } \quad \scalar{\varphi_i}{q^*_i} \ge 0, \qquad \forall i\in \Bcal(x^*).
  \end{equation}
\end{theorem}
\begin{proof}
  Since the optimal solution $x^*$ satisfies the optimality system \eqref{eq:strong_stationarity_TV}, we start from there. Equations \eqref{eq:strong_stationarity_primal_1}-\eqref{eq:strong_stationarity_primal_3} are just the constraints in \eqref{eq:bilevel_problem_kkt_reformulation}, while equations \eqref{eq:strong_stationarity_primal_4}-\eqref{eq:strong_stationarity_primal_5} follow immediatelly from \eqref{eq:strong_stationarity_TV_1} and \eqref{eq:strong_stationarity_TV_2}.

  To verify \eqref{eq:strong_stationarity_primal_8} let us note that $r^*_i =0$ and $\nu_i =0,$ for all $i \in \Acal(x^*).$ Consequently, equations \eqref{eq:strong_stationarity_TV_6} and \eqref{eq:strong_stationarity_TV_7} may be written as the linear system
    \begin{equation}
      \begin{pmatrix}
        \cos(\theta^*_i) &\sin(\theta^*_i)\\
        -\delta^*_i \sin(\theta^*_i) & \delta^*_i \cos(\theta^*_i)
      \end{pmatrix}
      \begin{pmatrix}
        \lambda_i^x\\ \lambda_i^y
      \end{pmatrix}
      =
      \begin{pmatrix}
        -\sigma_i \\ 0
      \end{pmatrix}, \qquad \forall i \in \Acal(x^*).
    \end{equation}
  Since by assumtion $\mathcal T(x^*)= \emptyset$, the system matrix becomes orthogonal (after dividing the second row by $\delta_i$). Consequently, since $\sigma_i=0$, for all $i \in \Acal(x^*)$, we get that $\lambda_i^x=0$ and $\lambda_i^y =0$, for all $i \in \Acal(x^*)$, which implies, thanks to \eqref{eq:strong_stationarity_TV_3}-\eqref{eq:strong_stationarity_TV_4}, that
  \begin{equation*}
    \left( \Kbb p \right)_i =0, \qquad \forall i \in \Acal(x^*).
  \end{equation*}

  For the characterization of the multiplier $\varphi$ on the inactive set, we consider the system resulting from equations \eqref{eq:strong_stationarity_TV_5} and \eqref{eq:strong_stationarity_TV_7}. Since $\delta_i^* = Q(\alpha^*)_i$, $r^*_i >0$ and $\gamma_i=0$ on $\Ical(x^*)$, we get
  \begin{equation}
    \begin{pmatrix}
      \cos(\theta^*_i) &\sin(\theta^*_i)\\
      -\sin(\theta^*_i) &\cos(\theta^*_i)
    \end{pmatrix}
    \begin{pmatrix}
      r_i \varphi_i^x\\ r_i \varphi_i^y
    \end{pmatrix}
    =
    \begin{pmatrix}
      0\\ - \delta^*_i \sin(\theta^*_i) \lambda_i^x+ \delta^*_i \cos (\theta^*_i) \lambda_i^y
    \end{pmatrix}, \qquad \forall i \in \Ical(x^*).
  \end{equation}
  Owing to the orthogonality of the system matrix, we obtain that
  \begin{align*}
    \begin{pmatrix}
      r^*_i \varphi_i^x\\ r_i \varphi_i^y
    \end{pmatrix}
    &=
    \begin{pmatrix}
      \cos(\theta^*_i) &-\sin(\theta^*_i)\\
      \sin(\theta^*_i) &\cos(\theta^*_i)
    \end{pmatrix}
    \begin{pmatrix}
      0\\ - \delta^*_i \sin(\theta^*_i) \lambda_i^x+ \delta^*_i \cos (\theta^*_i) \lambda_i^y
    \end{pmatrix}\\
    &= \delta^*_i
    \begin{pmatrix}
      \sin^2(\theta^*_i)\lambda_i^x- \sin(\theta^*_i) \cos (\theta^*_i)\lambda_i^y\\ -\sin(\theta^*_i)\cos (\theta^*_i) \lambda_i^x+ \cos^2 (\theta^*_i) \lambda_i^y
    \end{pmatrix}\\
    &= \delta^*_i
    \begin{pmatrix}
      \sin^2(\theta^*_i) &- \sin(\theta^*_i) \cos (\theta^*_i)\\
      - \sin(\theta^*_i) \cos (\theta^*_i) &\cos^2(\theta^*_i)
    \end{pmatrix}
    \begin{pmatrix}
      \lambda_i^x\\ \lambda_i^y
    \end{pmatrix}\\
    &= \delta^*_i
    \left[I -\begin{pmatrix}
      \cos^2(\theta^*_i) & \sin(\theta^*_i) \cos (\theta^*_i)\\
       \sin(\theta^*_i) \cos (\theta^*_i) &\sin^2(\theta^*_i)
    \end{pmatrix} \right]
    \begin{pmatrix}
          \lambda_i^x\\ \lambda_i^y
    \end{pmatrix}, && \forall i \in \Ical(x^*).
  \end{align*}
  Since, for all $i \in \Ical(x^*),$
  \begin{align*}
    (\Kbb u^*)_i (\Kbb u^*)_i^\intercal = (r^*_i)^2 \begin{pmatrix}
          \cos(\theta^*_i)\\ \sin(\theta^*_i)
    \end{pmatrix}
    \begin{pmatrix}
          \cos(\theta^*_i) &\sin(\theta^*_i)
    \end{pmatrix}
    = (r^*_i)^2 \begin{pmatrix}
      \cos^2(\theta^*_i) & \sin(\theta^*_i) \cos (\theta^*_i)\\
       \sin(\theta^*_i) \cos (\theta^*_i) &\sin^2(\theta^*_i)
    \end{pmatrix},
  \end{align*}
  we then get that
  \begin{equation*}
    r^*_i \varphi_i = \delta^*_i \left( I - \frac{(\Kbb u^*)_i (\Kbb u^*)_i^\intercal}{(r^*_i)^2} \right) \lambda_i, \qquad \forall i \in \Ical(x^*),
  \end{equation*}
  which can also be written as
  \begin{equation*}
    \|(\Kbb u^*)_i\| ~\varphi_j = Q(\alpha^*)_i \left( I- \frac{(\Kbb u^*)_i (\Kbb u^*)_i^\intercal}{  \|(\Kbb u^*)_i\|^2} \right) (\Kbb p)_i, \qquad \forall i \in \Ical(x^*).
  \end{equation*}

  In a similar manner, on the biactive set $\Bcal(x^*)$, $r^*_i =0$ and $\delta^*_i=Q(\alpha^*)_i$ and, from equations \eqref{eq:strong_stationarity_TV_6} and \eqref{eq:strong_stationarity_TV_7}, we obtain the system
  \begin{equation}
    \begin{pmatrix}
      \cos(\theta^*_i) &\sin(\theta^*_i)\\
      -\delta^*_i \sin(\theta^*_i) & \delta^*_i \cos(\theta^*_i)
    \end{pmatrix}
    \begin{pmatrix}
      \lambda_i^x\\ \lambda_i^y
    \end{pmatrix}
    =
    \begin{pmatrix}
      \nu_i- \sigma_i\\ 0
    \end{pmatrix}, \qquad \forall i \in \Bcal(x^*).
  \end{equation}
  Since for any orthogonal matrix $U$, $\|U x\| =\|x\|$, and using that $\mathcal T(x^*)=\emptyset$, it then follows that $\sigma_i=0, ~\forall i \in \Bcal(x^*),$ and
  \begin{equation}\label{eq:proof stationarity primal biactive multiplier nu}
    |\nu_i| = \|\lambda_i \|, \qquad \forall i \in \Bcal(x^*).
  \end{equation}
  Moreover, from \Cref{eq:strong_stationarity_TV_7} it follows that $\scalar{\lambda_i}{[-\sin(\theta_i^*) ~\cos(\theta_i^*)]}=0$, which implies, using the representation \eqref{eq:strong_stationarity_TV_3}-\eqref{eq:strong_stationarity_TV_4}, that
    \begin{equation}
      (\Kbb p)_i \in \spa(q^*_i).
    \end{equation}
  We next analyze the cases in $\Bcal(x^*)$:
  \begin{itemize}
    \item If $\nu_i=0$, then from \Cref{eq:proof stationarity primal biactive multiplier nu} and equations \eqref{eq:strong_stationarity_TV_3}-\eqref{eq:strong_stationarity_TV_4} it follows that
    $(\Kbb p)_i=0.$
    \item If $\gamma_i =0$, then from \Cref{eq:strong_stationarity_TV_5} we get that $\frac{1}{Q(\alpha^*)_i} \scalar{q^*_i}{\varphi_i} =0.$
    \item If $\nu_i \geq 0, ~\gamma_i \geq 0$, then we get from \Cref{eq:strong_stationarity_TV_5} that
    \begin{equation*}
      \scalar{\varphi_i}{q^*_i} = Q(\alpha^*)_i \cos(\theta^*_i) \varphi_i^x + Q(\alpha^*)_i \sin(\theta^*_i) \varphi_i^y = Q(\alpha^*)_i \gamma_i \geq 0.
    \end{equation*}
    Moreover, from \Cref{eq:strong_stationarity_TV_6} and \Cref{eq:proof stationarity primal biactive multiplier nu}, we then get that
    \begin{equation*}
      \scalar{q^*_i}{\left( \Kbb p \right)_i} = Q(\alpha^*)_i \|\left( \Kbb p \right)_i\|, \qquad \forall i \in \Bcal(x^*).
    \end{equation*}
  \end{itemize}
  Combining all cases, conditions \eqref{eq:strong_stationarity_primal_9} are obtained.

  For the characterization of $\nu$, let us first notice that, thanks to \eqref{eq:strong_stationarity_TV_10}, $\nu_i =0$, for all $i \in \Acal(x^*)$.
  On the inactive set, on the other hand, we know that $r^*_i = \|(\Kbb u^*)_i\|>0$. Consequently, 
  by multiplying \eqref{eq:strong_stationarity_TV_6} with $r^*_i$, we obtain
  \begin{equation*}
    r^*_i \nu_i = r^*_i \left( \cos (\theta^*_i) \lambda_i^x + \sin(\theta^*_i) \lambda_i^y \right) + r^*_i \sigma_i = \scalar{(\Kbb u^*)_i}{(\Kbb p)_i} + r^*_i \sigma_i, \qquad \forall i \in \Ical(x^*),
  \end{equation*}
  which implies that
  \begin{equation*}
    \nu_i = \frac{1}{\|(\Kbb u^*)_i\|} \scalar{(\Kbb u^*)_i}{(\Kbb p)_i} + \sigma_i, \qquad \forall i \in \Ical(x^*).
  \end{equation*}
  On $\Bcal(x^*)$, since by hypothesis $\Tcal(x^*) = \emptyset$, we know that $\sigma_i =0$ and $\delta^*_i = Q(\alpha^*)_i >0$. Consequently, from \eqref{eq:strong_stationarity_TV_6},
  \begin{equation*}
    \nu_i = \cos (\theta^*_i) \lambda_i^x + \sin(\theta^*_i) \lambda_i^y = \frac{1}{Q(\alpha^*)_i}\scalar{ q^*_i}{(\Kbb p)_i},\qquad \forall i \in \Bcal(x^*).
  \end{equation*}

The S-stationarity condition follows in a straightforward manner from the previous argumentation on the biactive set and \Cref{eq:strong_stationarity_TV_Strong}.
\end{proof}

\begin{remark}
  From the optimality system \eqref{eq:strong_stationarity_primal_variables} and the change of variables, system \eqref{eq:strong_stationarity_TV} may also be obtained by utilizing the orthogonal structure of the transformation matrix as in the proof of \Cref{thm:strong_stationarity_TV_original_variables}.
\end{remark}

\begin{remark}
  Also from \eqref{eq:strong_stationarity_primal_variables} it becomes clear that the multiplier $\sigma$ only plays a role on the index set $\{i \in \Ical(x^*): Q(\alpha^*)_i=0\}$. If this set turns out to be empty, the multiplier can be dismissed, along with the complementarity condition \eqref{eq:strong_stationarity_primal_6.5}.
\end{remark}

For MPCC problems it is well-known that the strong stationarity system \eqref{eq:strong_stationarity_TV} corresponds to the Karush-Kuhn-Tucker system of a locally around $x^*$ relaxed problem (see, e.g., \cite{scheel2000mathematical}), which in the case of problem \eqref{eq:bilevel_mpec_formulationTV} takes the following form:
\begin{mini*}
    {\substack{u,q,\alpha,r,\delta,\theta}}{\Jcal (u)}{}{}
    \addConstraint{\nabla_u \Dcal_0(u;f))+ \Kbb^\intercal q}{= 0}{}
    \addConstraint{(\Kbb u)_i}{= r_i [\cos(\theta_i),\sin(\theta_i)]^\intercal,}{\quad \forall i=1,\dots,m}
    \addConstraint{q_i}{= \delta_i [\cos(\theta_i),\sin(\theta_i)]^\intercal,}{\quad \forall i=1,\dots,m}{}
    \addConstraint{\delta_i}{\ge 0,}{\quad \forall i=1\dots,m}
    \addConstraint{r_i =0,\quad}{(Q(\alpha)_i-\delta_i) \ge 0,}{\quad \forall i \in \Acal(x^*)}{}
    \addConstraint{r_i \ge 0,\quad}{(Q(\alpha)_i-\delta_i) = 0,}{\quad \forall i \in \Ical(x^*)}{}
    \addConstraint{r_i \ge 0,\quad}{(Q(\alpha)_i-\delta_i) \ge 0,}{\quad \forall i \in \Bcal(x^*).}{}
\end{mini*}

In the case of problem \eqref{eq:bilevel_problem_kkt_reformulation}, using solely the original variables, a similar result, locally around $x^*$, can be verified for the relaxed problem:
  \begin{mini!}
     {\substack{u,q,\alpha}}{\Jcal(u)}{\label{eq:relaxed bilevel_mpec_formulationTV}}{}
     \addConstraint{\nabla_u \Dcal_0(u;f) + \Kbb^\intercal q}{= 0}
     \addConstraint{q_i}{= Q(\alpha)_i \frac{(\Kbb u)_i}{\|(\Kbb u)_i\|},}{\quad \forall i \in \Ical(x^*)}
     \addConstraint{(\Kbb u)_i}{= 0,}{\quad \forall i \in \Acal(x^*)}
     \addConstraint{\scalar{q_i}{(\Kbb u)_i}}{\geq 0,}{\quad \forall i \in \Bcal(x^*)}
     \addConstraint{\|q_i\|}{\le Q(\alpha)_i,}{\quad \forall i \in (\Acal(x^*) \cup \Bcal(x^*)) \cap \{ q_i^* \neq 0 \}}{\label{eq:relaxed bilevel_mpec_formulationTV_const_q}}{}
     \addConstraint{Q(\alpha)_i}{\geq 0,}{\quad \forall i \in \Ical(x^*).}
  \end{mini!}
Indeed, if $ x= ( u, q, \alpha)$ is an optimal solution to \eqref{eq:relaxed bilevel_mpec_formulationTV}, locally around $x^*$, it follows that $ q_i \neq 0$ if
$q^*_i \neq 0$ and $(\Kbb u)_i \neq 0$ if $(\Kbb u^*)_i  \neq 0$. Let us now assume that $\mathcal T(x^*)= \emptyset$ and introduce the Lagrangian
\begin{multline*}
  \Lcal(x,p,\phi,\gamma,\mu,\beta):= \Jcal(u)- \scalar{p}{\nabla_u \Dcal_0(u)+ \Kbb^\intercal q} + \sum_{\Ical(x^*)} \scalar{\phi_i}{q_i-Q(\alpha)_i \frac{(\Kbb u)_i}{\| (\Kbb u)_i \|}}\\
  - \sum_{\Acal(x^*)} \scalar{\mu_i}{(\Kbb u)_i} - \sum_{\Bcal(x^*)} \gamma_i \scalar{q_i}{(\Kbb u)_i}  + \sum_{\substack{\Acal(x^*)\cup \Bcal(x^*)}} \beta_i \left( \|q_i\|- Q(\alpha)_i \right)- \sum_{\Ical(x^*)} \sigma_i Q(\alpha)_i.
\end{multline*}

Taking the derivative with respect to $u$, in a direction $\delta_u$, yields
\begin{multline*}
  \nabla_u \Lcal(x,p,\phi,\gamma,\mu,\beta)[\delta_u] = \scalar{\nabla_u \Jcal(u)}{\delta_u} - \scalar{\nabla^2_{uu} \Dcal_0(u)^\intercal p}{\delta_u}\\ -
  \sum_{\Ical(x^*)} \scalar{\frac{Q(\alpha)_i}{\|(\Kbb u)_i\|} \left( I - \frac{(\Kbb u)_i (\Kbb u)_i^\intercal}{\|(\Kbb u)_i\|^2} \right) \phi_i}{(\Kbb \delta_u)_i}
  - \sum_{\Bcal(x^*)} \gamma_i \scalar{q_i}{(\Kbb \delta_u)_i} - \sum_{\Acal(x^*)}  \scalar{\mu_i}{(\Kbb \delta_u)_i}=0.
\end{multline*}
Setting
\begin{equation*}
  \varphi_i =
  \begin{cases}
      \frac{Q(\alpha)_i}{\|(\Kbb u)_i\|} \left( I - \frac{(\Kbb u)_i (\Kbb u)_i^\intercal}{\|(\Kbb u)_i\|^2} \right) \phi_i, & \forall i \in \Ical(x^*),\\
      \mu_i, & \forall i \in \Acal(x^*),\\
      \gamma_i q_i, & \forall i \in \Bcal(x^*),
  \end{cases}
\end{equation*}
we then obtain that
\begin{equation*}
  \nabla^2_{uu} \Dcal_0(u)^\intercal p + \Kbb^\intercal \varphi = \nabla_u \Jcal(u).
\end{equation*}

Taking the derivative of the Lagrangian with respect to $q$, in direction $\delta_q$, yields
\begin{multline*}
  \nabla_q \Lcal(x,p,\phi,\gamma,\mu,\beta)[\delta_q] = -\scalar{\Kbb p}{(\delta_q)_i} + \sum_{\Ical(x^*)} \scalar{\phi_i}{(\delta_q)_i}
  - \sum_{\Bcal(x^*)} \gamma_i \scalar{(\Kbb u)_i}{(\delta_q)_i}\\ + \sum_{\Acal(x^*) \cup \Bcal(x^*)}~ \beta_i \scalar{\frac{q_i}{\|q_i\|}}{(\delta_q)_i} =0,
\end{multline*}
which implies that
\begin{equation*}
  (\Kbb p)_i =
  \begin{cases}
      \phi_i, & \forall i \in \Ical(x^*),\\
      \beta_i \frac{q_i}{\|q_i\|}, & \forall i \in \Acal(x^*),\\
      \beta_i \frac{q_i}{\|q_i\|} - \gamma_i (\Kbb u)_i, & \forall i \in \Bcal(x^*),\\
      0, &\text{elsewhere}.
  \end{cases}
\end{equation*}

On the set $\Acal(x^*)$, since $0 < \|q_i^*\| < Q(\alpha)_i$ and $q_i$ is in a neighborhood of $q_i^*$, we get that $0 < \|q_i\| < Q(\alpha)_i$. Thanks to the complementarity with respect to the KKT multiplier, $\beta_i=0$ on this set, which further implies that
\begin{equation*}
  (\Kbb p)_i = 0, \qquad \forall i \in \Acal(x^*).
\end{equation*}

On the set $\Bcal(x^*)$,
\begin{itemize}
  \item if $(\Kbb u)_i = 0$, it follows that $(\Kbb p)_i$ and $q_i$ are collinear, which implies that $\scalar{q_i}{(\Kbb p)_i}= \|q_i\| \|(\Kbb p)_i\|$. If $\|q_i\| < Q(\alpha)_i$, then by complementarity $\beta_i=0$ and, therefore, $(\Kbb p)_i = 0$. Otherwise, $\|q_i\| = Q(\alpha)_i$.
  \item if $(\Kbb u)_i \neq 0$, then $\scalar{q_i}{(\Kbb p)_i} = \beta_i
  \|q_i\| - \gamma_i \scalar{q_i}{(\Kbb u)_i}$. Thanks to the complementarity with respect to the multiplier $\gamma_i$, it then follows that $(\Kbb p)_i=\beta_i \frac{q_i}{\|q_i\|}$. Due to the collinearity, it also follows that $\scalar{q_i}{(\Kbb p)_i}= \|q_i\| \|(\Kbb p)_i\|$. If $\|q_i\| < Q(\alpha)$,
  then $\beta_i=0$ and $(\Kbb p)_i=0$. Otherwise, $\|q_i\|=Q(\alpha)_i$. Together we get that $\scalar{q_i}{(\Kbb p)_i}= Q(\alpha)_i \|(\Kbb p)_i\|$.
\end{itemize}
Consequently, we obtain
\begin{equation} \label{eq: OS_relaxed_scalar_1}
  \scalar{q_i}{(\Kbb p)_i}= Q(\alpha)_i \|(\Kbb p)_i\|, \qquad \forall i \in \Bcal(x^*) \cap \{q_i^* \neq 0\}.
\end{equation}
In addition, $\scalar{\varphi_i}{q_i} = \gamma_i \|q_i\|^2 \geq 0, ~ \forall i \in \Bcal(x^*).$

Taking the derivative of the Lagrangian with respect to $\alpha$ yields
\begin{equation*}
  \nabla_\alpha \Lcal(x,p,\phi,\gamma,\mu,\beta) = -\chi_{\Ical(x^*)} \nabla_\alpha Q(\alpha)_i \left( \scalar{\phi_i}{\frac{(\Kbb u)_i}{\|(\Kbb u)_i\|}}+ \sigma_i \right)- \chi_{\Acal(x^*) \cup \Bcal(x^*) \backslash q_i^* =0} \beta_i \nabla_\alpha Q(\alpha)_i =0.
\end{equation*}
On the set $\Acal(x^*) \cap \{q_i^* \neq 0\}$, it follows by complementarity that $\beta_i=0$. On $\Bcal(x^*) \cap \{q_i^* \neq 0\}$ we get, with similar arguments as above, that $(\Kbb p)_i=\beta_i \frac{q_i}{\|q_i\|}$ and $\scalar{q_i}{(\Kbb p)_i}= Q(\alpha)_i \|(\Kbb p)_i\|$, which implies that $\beta_i = \frac{1}{Q(\alpha)_i}\scalar{q_i}{(\Kbb p)_i}.$ Setting
\begin{equation*}
  \nu_i =
  \begin{cases}
      \frac{1}{\|(\Kbb u)_i\|} \scalar{(\Kbb u)_i}{(\Kbb p)_i}+ \sigma_i, & \forall i \in \Ical(x^*)\\
      0, & \forall i \in \Acal(x^*)\\
      \frac{1}{Q(\alpha)_i}\scalar{q_i}{(\Kbb p)_i}, & \forall i \in \Bcal(x^*),
  \end{cases}
\end{equation*}
we then obtain that $\nabla_\alpha Q(\alpha) \nu =0$,
and the equivalence is verified.

\subsubsection{Case with scalar parameter}
In the case of a scalar weight, it is known, under weak data conditions \cite{reyes2015a}, that the optimal parameter $\alpha^* >0$. Since in this case $Q(\alpha)= \alpha \mathbf{1}$, it then follows that $\sigma_i =0,$ for all $i \in \Ical(x^*)$. Consequently, from \Cref{eq:strong_stationarity_primal_9} and \Cref{eq:strong_stationarity_primal_11},
\begin{equation*}
  \nu_i = \frac{1}{\alpha^*} \scalar{q^*_i}{(\Kbb p)_i}, \qquad \forall i \in \Ical(x^*) \cup \Bcal(x^*).
\end{equation*}
The strong stationary system simplifies to
\begin{subequations} \label{eq:strong_stationarity_scalar_case}
\begin{align}
    & \nabla_u \Dcal_0(u^*) + \Kbb^\star q^* = 0\\
    & \scalar{q^*_i}{(\Kbb u^*)_i} = \alpha^*  \|(\Kbb u^*)_i\|,\; && \forall i=1,\dots,n\\
    & \|q^*_i\| \le \alpha^*,\; && \forall i=1,\dots,n\\
    & \nabla_{uu}^2 \Dcal_0 (u^*)^\intercal p + \Kbb^\star \varphi = \nabla_u \Jcal (u^*),\\
    & \sum_{i \in \Ical(x^*) \cup \Bcal(x^*)} \scalar{q^*_i}{(\Kbb p)_i} =0,\\
    & \|(\Kbb u^*)_i\| ~\varphi_i = \alpha^* \left( I- \frac{(\Kbb u^*)_i (\Kbb u^*)_i^\intercal}{\|(\Kbb u^*)_i\|^2} \right) (\Kbb  p)_i, \; && \forall i \in \Ical(x^*),\\
    & \left( \Kbb p \right)_i =0, \; && \forall i \in \Acal(x^*),\\
    & \scalar{q^*_i}{(\Kbb p)_i} = \alpha^* \|( \Kbb p )_i\|, \; && \forall i \in \Bcal(x^*),\\
    & \scalar{\varphi_i}{q^*_i} \ge 0,\; && \forall i\in \Bcal(x^*).
\end{align}
\end{subequations}

\subsubsection{Case with scale-dependent parameter}
In the case of a spatial-dependent parameter $\alpha \in \R^n$, the parameter function $Q$ is just the identity matrix. Consequently, \Cref{eq:strong_stationarity_primal_5} yields $\nu_i =0, \, \forall i.$
From \Cref{eq:strong_stationarity_primal_11},
\begin{equation*}
  \frac{1}{\alpha^*_i}\scalar{q^*_i}{(\Kbb p)_i}, \quad \forall i \in \Ical(x^*) \cup \Bcal(x^*).
\end{equation*}
Moreover, combining \Cref{eq:strong_stationarity_primal_7} and \Cref{eq:strong_stationarity_primal_11}, we get that
\begin{equation*}
  \|(\Kbb u^*)_i\| ~\varphi_i = \alpha^*_i (\Kbb  p)_i, \quad \forall i \in \Ical(x^*).
\end{equation*}
Combining the latter with \Cref{eq:strong_stationarity_primal_12} we also obtain that
\begin{equation*}
  \scalar{\varphi_i}{q_i^*} =\alpha_i^* \|\varphi_i\|, \quad \forall i \in \Ical(x^*).
\end{equation*}
Consequently, the S-stationarity system takes the following form:
\begin{subequations} \label{eq:strong_stationarity_scale_dependent}
\begin{align}
    & \nabla_u \Dcal_0(u^*) + \Kbb^\intercal q^* = 0\\
    & \scalar{q^*_i}{(\Kbb u^*)_i} = \alpha^*_i  \|(\Kbb u^*)_i\|,\; && \forall i=1,\dots,n \\
    &\|q^*_i\| \le \alpha^*_i,\; && \forall i=1,\dots,n \\
    & \nabla_{uu}^2 \Dcal_0 (u^*)^\intercal p + \Kbb^\intercal \varphi = \nabla_u \Jcal (u^*), \\
    & \|(\Kbb u^*)_i\| ~\varphi_i = \alpha^*_i (\Kbb  p)_i, \; && \forall i \in \Ical(x^*)\\
    & \left( \Kbb p \right)_i =0, \; && \forall i \in \Acal(x^*) \cup \Bcal(x^*)\\
    & \scalar{\varphi_i}{q_i^*} =\alpha_i^* \|\varphi_i\|, \; && \forall i \in \Ical(x^*)\\
    & \scalar{\varphi_i}{q_i^*} \ge 0,\;&& \forall i \in \Bcal(x^*).
\end{align}
\end{subequations}
For Gaussian noise, i.e., $\Dcal_0(u)= \|u-f\|^2$ and quadratic loss function, if certain conditions on the total variation of the noisy and ground truth images are fulfilled (see \cite{reyes2015a} for details), we also know that $\alpha_i >0$, for all $i$.


\subsection{Remark on the Infinite-Dimensional Case}\label{sec:infinite}
Let us now consider the infinite-dimensional total variation denoising model given by
\begin{mini!}
  {u}{\frac{\mu}{2} \int_\Omega |\nabla u|^2~dx + \frac{1}{2} \int_\Omega |u-f|^2 ~dx + \alpha \int_\Omega |\nabla u| ~dx,}{\label{eq: inf dim TV lower level}}{}
\end{mini!}
where $f \in L^{\infty}(\Omega)$, $\Omega \subset \R^2$ is a convex domain and $0 < \mu \ll 1$ is an artificial diffusion parameter. In this case the unique solution to the denoising problem belongs to the Sobolev space $H^1(\Omega)$. Moreover, the solution is characterized by the existence of a dual multiplier $q \in L^2(\Omega;\R^2)$, such that the following extremality conditions are satisfied:
\begin{subequations}\label{eq:inf dim primal_dual_stationarity_lower_level}
    \begin{align}
        - \mu \Delta u + u + \operatorname{div} q &= f,\\
        \scalar{q(x)}{\nabla u(x)}&=  \alpha \|\nabla u(x)\|, && \text{ a.e. in } \Omega\\
        \|q(x)\|&\le \alpha, && \text{ a.e. in } \Omega
    \end{align}
\end{subequations}

Although, intuitively, a similar change of variables as in problem \eqref{eq:bilevel_mpec_formulation} may be used in this case, a careful treatment must be carried out due to its infinite-dimensional character. Thanks to the convexity of $\Omega$, extra regularity results hold for this problem. In particular we get that $q \in [L^{\infty}(\Omega)]^2$ and $\nabla u \in [L^{\infty}(\Omega)]^2$, which allows us to define
$r(x) := \|\nabla u(x)\| \in L^{\infty}(\Omega)$ and $\delta(x) := \|q(x)\| \in L^{\infty}(\Omega)$. By using the collinearity condition, we may introduce 
\begin{equation*}
  \theta(x):=
  \begin{cases}
    \arccos \left( \partial_x u(x)/r(x) \right) &\text{if } \partial_y u(x) \geq 0, ~r(x) \neq 0,\\
    -\arccos \left( \partial_x u(x)/r(x) \right) &\text{if } \partial_y u(x) < 0, ~r(x) \neq 0,\\
    \arccos \left( q^x(x)/\delta(x) \right) &\text{if } q^y(x) \geq 0, ~\delta(x) \neq 0, ~r(x)=0,\\
    -\arccos \left( q^x(x)/\delta(x) \right) &\text{if } q^y(x) < 0, ~\delta(x) \neq 0, ~r(x)=0,\\
    \text{undefined} &\text{if } \delta(x) = 0, ~r(x)=0.
  \end{cases}
\end{equation*}
Thanks to the boundedness of both $r(x)$ and $\delta(x)$ and the continuity of $\arccos(\cdot)$, it follows that $\theta(x)$ is also in $L^\infty(\Omega)$ and
\begin{align*}
   \nabla u(x) & = r(x) [\cos(\theta (x)),\sin(\theta (x))]^\intercal, && \text{ a.e. in } \Omega\\
   q(x) & = \delta (x) [\cos(\theta(x)),\sin(\theta(x))]^\intercal,&& \text{ a.e. in } \Omega.
\end{align*}

Similarly as for the finite-dimensional case, the following complementarity conditions are then fulfilled by the auxiliary variables:
\begin{align*}
    &  \delta(x) \ge 0, && \text{ a.e. in } \Omega,\\
    &  0\le r(x) \perp (\alpha -\delta(x)) \ge 0, && \text{ a.e. in } \Omega.
\end{align*}
With this reformulation it is possible to write the bilevel problem as follows:
\begin{mini}[3]
    {\substack{u,q,\lambda,r,\delta, \theta}}{\Jcal (u)}{\label{eq:bilevel_mpec_formulation_ininite}}{}
    \addConstraint{- \mu \Delta u + u+ \operatorname{div} q}{= f}
    \addConstraint{\nabla u(x)}{= r(x) [\cos(\theta (x)),\sin(\theta (x))]^\intercal,}{\;\quad \text{ a.e. in } \Omega}
    \addConstraint{q(x)}{= \delta (x) [\cos(\theta(x)),\sin(\theta(x))]^\intercal,}{\;\quad \text{ a.e. in } \Omega}
    \addConstraint{\delta(x)}{\ge 0,}{\;\quad \text{ a.e. in } \Omega}
    \addConstraint{0\le r(x)}{\perp (\alpha -\delta (x)) \ge 0,}{{\;\quad \text{ a.e. in } \Omega}.}
\end{mini}
The constraints in \eqref{eq:bilevel_mpec_formulation_ininite} involve pointwise inequalities on the Euclidean norm of the gradient of $u$ and also on the Euclidean norm of the dual multiplier $q$, which are, in addition, in complementarity to each other. Although there are several contributions on PDE-constrained optimization problems involving state constraints, pointwise constraints on the gradient of the state are particularly difficult to handle; even more so if the state variables are also in pointwise complementarity. Problem \eqref{eq:bilevel_mpec_formulation_ininite} is indeed a very challenging one that requires a detailed treatment, which is beyond the scope of this paper.

\section{Second-Order Optimality Conditions}\label{sec:second}
Let us now turn to second-order optimality conditions and let us summarize first some known results from the literature (see \cite{guo2013second} for further details). We start by defining the Lagrangian for the MPCC problem \eqref{eq:MPCC} as follows:
\begin{equation}
  \mathcal L_{MPCC}(x,\lambda, \mu, \gamma, \nu): = \ell(x)+ h(x)
^\intercal \lambda + g(x)^\intercal \mu -M(x)^\intercal \gamma - N(x)^\intercal \nu.
\end{equation}


\begin{definition}
  Let $x^* \in X$ be feasible for \eqref{eq:MPCC}. The M-multiplier strong second-order sufficient condition (M-SSOSC) holds at $x^* \in X$ iff, for every M-multiplier $(\lambda, \mu, \gamma, \nu)$,
  \begin{equation}\label{eq: M-SSOSC}
    d^\intercal \nabla_{xx}^2 \mathcal L_{MPCC}(x^*,\lambda, \mu, \gamma, \nu) d >0, \quad \forall d \in \mathcal C_{MPCC}(x^*) \backslash \{0\},
  \end{equation}
  where $\mathcal C_{MPCC}(x^*):=L_X^{MPCC}(x^*) \cap \{d: \nabla \ell(x^*)^\intercal d \leq 0 \}$.
\end{definition}

\begin{theorem}[Guo, Lin, Ye \cite{guo2013second}]
  Let $x^*$ be an M-stationary point of the general MPCC \eqref{eq:MPCC}. Suppose that both the MPEC-RCPLD and M-SSOSC hold at $x^*$. Then there exists a constant $\delta >0$ such that, if $x \in B_\delta (x^*) \cap X$ and there is an M-multiplier for $x$, there must hold $x=x^*$. Moreover, if $x^*$ is an S-stationary point of \eqref{eq:MPCC}, then there exists a neighborhood $V$ of $x^*$ containing no other S-stationary point. 
\end{theorem}

Turning back to problem \eqref{eq:bilevel_mpec_formulationTV}, let us introduce the corresponding MPCC-Lagrangian:
\begin{multline}
  \mathcal L_{MPCC}(x,\lambda, \mu, \gamma, \nu)= \Jcal (u) + \scalar{p}{-\nabla_u \Dcal_0 (u)- \Kbb^\intercal q}\\ + \scalar{\varphi^x}{-K_x u + r \circ \cos(\theta)} + \scalar{\varphi^y}{-K_y u + r \circ \sin(\theta)}\\
  + \scalar{\lambda^x}{q^x - \delta \circ \cos(\theta)} + \scalar{\lambda^y}{q^y - \delta \circ \sin(\theta)} - \sigma^\intercal \delta -\gamma^\intercal r - \nu^\intercal (Q(\alpha)- \delta).
\end{multline}
Since M- and S-stationarity conditions hold under suitable assumptions (see \Cref{thm: strong stationarity MPCC}), second-order sufficient conditions may be verified under strong convexity of the Lagrangian for critical directions. In the next theorem we study under which conditions such sufficiency holds.
\begin{theorem} \label{thm: SSC}
     Let $x^*$ be an M-stationary point of \eqref{eq:bilevel_mpec_formulationTV} such that MPEC-RCPLD holds. If, for every M-multiplier and non-vanishing critical direction $d \in L_X^{MPCC}(x^*) \cap \{d: \nabla \Jcal(u^*)^\intercal d_u \leq 0 \}$,
     \begin{multline*}
       \langle \nabla_{u u}^2 \Jcal (u^*) d_u, d_u \rangle - \langle [\nabla_u (\nabla_{u u}^2 \Dcal (u^*)^\intercal p)]^\intercal d_u, d_u \rangle\\
       -  \langle [\nabla_\alpha (\nabla_\alpha Q(\alpha^*))^\intercal \nu)]^\intercal d_\alpha, d_\alpha \rangle
       + 2 \langle - \varphi^x \circ K_y d_u + \varphi^y \circ K_x d_u, d_\theta \rangle_{\Ical^* \cup \Bcal^*}
       \\+ 2 \langle (\lambda^x \circ \sin (\theta^*)- \lambda^y \circ \cos (\theta^*)) \circ \nabla_\alpha Q(\alpha^*)_i^\intercal d_\alpha, d_\theta \rangle_{\Ical^*} + \langle \delta \circ \nu \circ d_\theta, d_\theta \rangle_{\Ical^* \cup \Bcal^*} >0,
     \end{multline*}
     then $x^*$ is the locally unique M-stationary point of \eqref{eq:bilevel_mpec_formulationTV}.
\end{theorem}
\begin{proof}
  To prove the result, we need to verify that for every set of M-multipliers $(p,\varphi, \lambda, \rho, \sigma, \gamma, \nu)$ and every non-vanishing critical direction $d \in L_X^{MPCC}(x^*) \cap \{d: \nabla \Jcal(u^*)^\intercal d_u \leq 0 \}$,
  \begin{equation*}
    d^\intercal \nabla_{xx}^2 \Lcal (x^*, p,\varphi, \lambda, \rho, \sigma, \gamma, \nu) d >0.
  \end{equation*}
  The critical directions satisfy the following equations for the equality constraints
  \begin{align}
    &-\nabla_{uu}^2 \Dcal (u^*)^\intercal d_u - K_x^\intercal d_{q^x} - K_y^\intercal d_{q^y}=0,\\
    & - K_x d_u + \cos (\theta^*) \circ d_r - r^* \circ \sin(\theta^*) \circ d_\theta =0, \label{eq: critical directions 2}\\
    & - K_y d_u + \sin (\theta^*) \circ d_r + r^* \circ \cos(\theta^*) \circ d_\theta =0, \label{eq: critical directions 3}\\
    & d_{q^x} - \cos (\theta^*) \circ d_\delta + \delta^* \circ \sin(\theta^*) \circ d_\theta =0,\\
    & d_{q^y} - \sin (\theta^*) \circ d_\delta - \delta^* \circ \cos(\theta^*) \circ d_\theta =0,
  \end{align}
  the relations for the inequality restrictions
  \begin{align}
    & d_{\delta,i} \geq 0, && \text{if } \delta_i=0,
  \end{align}
  the complementarity conditions
  \begin{align}
    &\nabla Q(\alpha^*)_i^\intercal d_\alpha = d_{\delta,i}, && \text{if } i \in \Ical(x^*), \label{eq: critical directions 8}\\
    &d_{r,i} = 0, && \text{if } i \in \Acal(x^*) \label{eq: critical directions 9}\\
    &0 \leq \left( \nabla Q(\alpha^*)_i^\intercal d_\alpha - d_{\delta,i} \right) \perp d_{r,i} \geq 0, && \text{if } i \in \Bcal(x^*), \label{eq: critical directions 10}
  \end{align}
  and
  \begin{equation}
    \nabla_u \mathcal J(u^*)^\intercal d_u \leq 0.
  \end{equation}

  For the first and second derivatives of the Lagrangian with respect to $u$, in direction $d_u$, we obtain
  \begin{align*}
    & \nabla_{u} \Lcal(z^*) [d_u] = \langle \nabla_u \Jcal (u^*) - \nabla_{u u}^2 \Dcal (u^*)^\intercal p - K_x^\intercal \varphi^x - K_y^\intercal \varphi^y, d_u \rangle,\\
    & \nabla_{u u}^2 \Lcal(z^*) [d_u]^2 = \langle \nabla_{u u}^2 \Jcal (u^*) d_u, d_u \rangle - \langle [\nabla_u (\nabla_{u u}^2 \Dcal (u^*)^\intercal p)]^\intercal d_u, d_u \rangle,
  \end{align*}
  where $z^*:=(x^*,p,\varphi, \lambda, \rho, \sigma, \gamma, \nu).$
  The first and second derivatives with respect to $\alpha$ are given by
  \begin{align*}
    & \nabla_{\alpha} \Lcal(z^*) [d_\alpha] = - \langle \nabla_\alpha Q(\alpha^*)^\intercal \nu, d_\alpha \rangle,\\
    & \nabla_{\alpha \alpha}^2 \Lcal(z^*) [d_\alpha]^2 = - \langle [\nabla_\alpha (\nabla_\alpha Q(\alpha^*))^\intercal \nu)]^\intercal d_\alpha, d_\alpha \rangle,
  \end{align*}
  The first derivatives of the Lagrangian with respect to $r$ and $\delta$ are given by
  \begin{align*}
    & \nabla_r \Lcal (z^*)[d_r] = \langle \varphi^x \circ \cos (\theta^*) + \varphi^y \circ \sin(\theta^*) - \gamma, d_r \rangle\\
    & \nabla_\delta \Lcal (z^*)[d_\delta] = - \langle \lambda^x \circ \cos (\theta^*) + \lambda^y \circ \sin(\theta^*) + \sigma - \nu, d_\delta \rangle,
  \end{align*}
  respectively. Deriving the latters with respect to $\theta$ we get
  \begin{align}
    & \nabla_{r \theta}^2 \Lcal(z^*) [d_r,d_\theta] = \langle (-\varphi^x \circ \sin (\theta^*)+ \varphi^y \circ \cos (\theta^*)) \circ d_\theta, d_r \rangle, \label{eq: mixed second derivative r}\\
    & \nabla_{\delta \theta}^2 \Lcal(z^*) [d_\delta,d_\theta] = \langle (\lambda^x \circ \sin (\theta^*)- \lambda^y \circ \cos (\theta^*)) \circ d_\theta, d_\delta \rangle. \label{eq: mixed second derivative delta}
  \end{align}
  In addition we get that $\nabla_{\theta r}^2 \Lcal(z^*) [d_\theta, d_r]= \nabla_{r \theta}^2 \Lcal(z^*) [d_r,d_\theta]$ and $\nabla_{\theta \delta}^2 \Lcal(z^*) [d_\theta,d_\delta] =\nabla_{\delta \theta}^2 \Lcal(z^*) [d_\delta,d_\theta]$.

  From \Cref{eq: mixed second derivative r} we obtain, thanks to equations \eqref{eq: critical directions 2}, \eqref{eq: critical directions 3} and \eqref{eq: critical directions 9}, that
  \begin{equation*}
    \nabla_{r \theta}^2 \Lcal(z^*) [d_r,d_\theta] =\langle r^* \circ (\varphi^x \circ \cos (\theta^*)+ \varphi^y \circ \sin (\theta^*)) \circ d_\theta, d_\theta \rangle_{\Ical^*} + \langle - \varphi^x \circ K_y d_u + \varphi^y \circ K_x d_u, d_\theta \rangle_{\Ical^* \cup \Bcal^*}.
  \end{equation*}
  Using additionally \Cref{eq:strong_stationarity_TV_5} and \Cref{eq:strong_stationarity_TV_11} we then get that
  \begin{equation} \label{eq: reformulated mixed second derivative r}
    \nabla_{r \theta}^2 \Lcal(z^*) [d_r,d_\theta] = \langle - \varphi^x \circ K_y d_u + \varphi^y \circ K_x d_u, d_\theta \rangle_{\Ical^* \cup \Bcal^*}.
  \end{equation}

  From \Cref{eq: mixed second derivative delta} we get, using the fact that $\lambda^x|_{\Acal^*} =0$, $\lambda^y|_{\Acal^*} =0$ and equations \eqref{eq:strong_stationarity_TV_7} and \eqref{eq: critical directions 8}, that
  \begin{equation}
    \nabla_{\delta \theta}^2 \Lcal(z^*) [d_\delta,d_\theta] = \langle (\lambda^x \circ \sin (\theta^*)- \lambda^y \circ \cos (\theta^*)) \circ \nabla Q(\alpha^*)_i^\intercal d_\alpha, d_\theta \rangle_{\Ical^*}.
  \end{equation}

  The first and second derivatives with respect to $\theta$ are given by
  \begin{align*}
    &  \nabla_{\theta} \Lcal(z^*) [d_\theta] =  \left\langle -\varphi^x \circ r \circ \sin (\theta^*) + \varphi^y \circ r \circ \cos(\theta^*) + \lambda^x \circ \delta \circ \sin(\theta^*) - \lambda^y \circ \delta \circ \cos (\theta^*), d_\theta \right\rangle,\\
    & \nabla_{\theta \theta}^2 \Lcal(z^*) [d_\theta]^2 = \left\langle (-\varphi^x \circ r \circ \cos (\theta^*) - \varphi^y \circ r \circ \sin(\theta^*) + \lambda^x \circ \delta \circ \cos (\theta^*) + \lambda^y \circ \delta \circ \sin (\theta^*)) \circ d_\theta, d_\theta \right\rangle,
  \end{align*}
  which, using \Cref{eq:strong_stationarity_TV_5} and \Cref{eq:strong_stationarity_TV_6}, implies that
  \begin{equation*}
    \nabla_{\theta \theta}^2 \Lcal(z^*) [d_\theta]^2 = \left\langle (- \gamma \circ r^* + (\nu- \sigma) \circ \delta^*) \circ d_\theta, d_\theta \right\rangle.
  \end{equation*}
  Since $\gamma_i=0$, for all $i \in \Ical(x^*)$, and $r^*_i=0$, for all $i \in \Acal(x^*) \cup \Bcal(x^*)$, the first term on the right hand side vanishes. Moreover, since $0 \leq \delta^* \perp \sigma \geq 0$, we obtain
  \begin{equation*}
    \nabla_{\theta \theta}^2 \Lcal(z^*) [d_\theta]^2 = \sum_{i \in \Ical^* \cup \Bcal^*} \nu_i \delta^*_i d_{\theta,i}^2 \geq 0.
  \end{equation*}
  All remaining second partial derivatives are equal to zero.

  Altogether we then obtain that
  \begin{align*}
    d^T \nabla_{xx}^2 & \Lcal (x^*, p,\varphi, \lambda, \rho, \sigma, \gamma, \nu) d =  \langle \nabla_{u u}^2 \Jcal (u^*) d_u, d_u \rangle - \langle [\nabla_u (\nabla_{u u}^2 \Dcal (u^*)^\intercal p)]^\intercal d_u, d_u \rangle\\
    & -  \langle [\nabla_\alpha (\nabla_\alpha Q(\alpha^*))^\intercal \nu)]^\intercal d_\alpha, d_\alpha \rangle
    + 2 \langle - \varphi^x \circ K_y d_u + \varphi^y \circ K_x d_u, d_\theta \rangle_{\Ical^* \cup \Bcal^*}\\
    &+ 2 \langle (\lambda^x \circ \sin (\theta^*)- \lambda^y \circ \cos (\theta^*)) \circ \nabla_\alpha Q(\alpha^*)_i^\intercal d_\alpha, d_\theta \rangle_{\Ical^*} + \langle \delta \circ \nu \circ d_\theta, d_\theta \rangle_{\Ical^* \cup \Bcal^*}.
  \end{align*}
  which completes the proof.
\end{proof}

\section{General Bilevel Problem}\label{sec:general}
The techniques developed so far for the bilevel total variation case may be extended to the general problem \eqref{eq:bilevel_mpec_formulation_general}, involving several data fidelity terms, as well as different first- and second-order sparsity based regularizers.
In the next result, existence of Lagrange multipliers for the general problem \eqref{eq:bilevel_mpec_formulation_general} is verified, and \textbf{M-stationarity conditions} are derived.

To formulate the result, let us introduce the following active, inactive and biactive sets for the Euclidean norm regularizers:
\begin{align*}
  \Acal_j^\alpha (x)&{}:=\{i=1,\dots, m_j: r_{j,i} =0, Q_j(\alpha_j)_i> \delta_{j,i} \}, \quad j=1, \dots, M,\\
  \Ical_j^\alpha (x)&{}:=\{i=1,\dots, m_j: r_{j,i} >0, Q_j(\alpha_j)_i= \delta_{j,i}\}, \quad j=1, \dots, M,\\
  \Bcal_j^\alpha (x)&{}:=\{i=1,\dots, m_j:r_{j,i} =0, Q_j(\alpha_j)_i= \delta_{j,i}\}, \quad j=1, \dots, M,
\end{align*}
and the active, inactive and biactive sets for the Frobenius norm regularizers:
\begin{align*}
  \Acal_j^\beta (x)&{}:=\{i=1,\dots, n_j:\rho_{j,i} =0, S_j(\beta_j)_i> \tau_{j,i} \}, \quad j=1, \dots, N,\\
  \Ical_j^\beta (x)&{}:=\{i=1,\dots, n_j:\rho_{j,i} >0, S_j(\beta_j)_i= \tau_{j,i}\}, \quad j=1, \dots, N,\\
  \Bcal_j^\beta (x)&{}:=\{i=1,\dots, n_j:\rho_{j,i} =0, S_j(\beta_j)_i= \tau_{j,i}\}, \quad j=1, \dots, N.
\end{align*}
We assume hereafter that all row vectors of the matrices $K_j^x, K_j^y$, $j=1, \dots, M$, and $E_j^x, E_j^y, E_j^z$, $j=1, \dots, N$, are non-zero.

\begin{theorem} \label{thm: strongg stationarity general problem}
  Let $x^*=(u^*,c^*,q^*,\Lambda^*, \lambda^*, \sigma^*, \alpha^*, \beta^*,r^*,\delta^*,\rho^*, \tau^*, \theta^*, \phi^*, \varphi^*)$ be a local optimal solution to \eqref{eq:bilevel_mpec_formulation_general}.
  Assume that the following hold:\vspace{0.5em}
  \begin{itemize} \setlength\itemsep{0.5em}
    \item[(H1)] For $j=1, \dots, M, ~\{i: r^*_{j,i}=\delta^*_{j,i} =0 \} = \emptyset$, 
    \item[(H2)] For $j=1, \dots, N, ~~\{i: \rho^*_{j,i}=\tau^*_{j,i} =0 \} = \emptyset$,  
    \item[(H3)] For $j=1, \dots, M,~[\nabla_{\alpha_j} Q_j(\alpha^*_j)_i]_{\{\delta^*_{j,i} = Q_j(\alpha_j^*)_i \}} \zeta =0 \implies  \zeta_i =0, ~\forall i \in \Bcal_j^\alpha (x^*)$,
    \item[(H4)] For $j=1, \dots, N,~[\nabla_{\beta_j} S_j(\beta^*_j)_i]_{\{\tau^*_{j,i} = S_j(\beta_j^*)_i \}} \tilde \zeta =0 \implies  \tilde \zeta_i =0, ~\forall i \in \Bcal_j^\beta (x^*)$, 
    \item[(H5)] For $j=1, \dots, K,~ \exists w \in \R^{|\lambda_j|}: \nabla_{\lambda_j} P_j(\lambda^*_j)^\intercal w \leq0, ~\nabla_{\lambda_j} P_j(\lambda^*_j)^\intercal w \neq 0,$ 
    \item[(H6)] For $j=1, \dots, L,~ \exists \tilde w \in \R^{|\sigma_j|}: \nabla_{\sigma_j} R(\sigma^*_j)^\intercal \tilde w \leq 0, \nabla_{\sigma_j} R(\sigma^*_j)^\intercal \tilde w \neq 0.$
    \item[(H7)] For any $x$ in a neighbourhood of $x^*$, $(\bar \psi, \bar \xi)= \mathbf 0$ is the unique solution to the system 
    \begin{subequations}
      \begin{align*}
        & \sum_{j=1}^M \Kbb_j^\star \psi_j + \sum_{j=1}^N \Ebb_j^\star \xi_j =0,\\
        & \chi_{\Ical^\alpha_j(x^*)} \psi_{j} =0, && j=1, \dots, M,\\
        & \chi_{\Acal^\alpha_j(x^*) \cup \Bcal^\alpha_j(x^*)} \left( \Cos(\theta_j) \psi_{j}^{x}+ \Sin(\theta_j) \psi_{j}^{y} \right) = 0,  && j=1, \dots,M,\\
        &\chi_{\Ical^\beta_j(x^*)} \xi_{j} =0, && j=1, \dots, N,\\
        & \chi_{\Acal^\beta_j(x^*) \cup \Bcal^\beta_j(x^*)} \left( \Sin(\phi_j) \left( \Cos(\varphi_j) \xi_j^x+ \Sin(\varphi_j) \xi_j^y \right) + \Cos(\phi_j) \xi_j^z \right) = 0,  && j=1, \dots,N.
      \end{align*}
    \end{subequations}    
  \end{itemize} \vspace{0.5em}
  Then there exist Lagrange multipliers
  $(p,\pi, \zeta,\mu_{\mathbb B}^+,\mu_{\mathbb B}^-, \mu_{\sigma}^+,,\mu_{\sigma}^-,\mu_{c},\mu_R,\mu_P,\mu_\tau,\mu_\delta,\nu^\alpha, \nu^\beta, \gamma^\alpha, \gamma^\beta)$
  such that, together with equations \eqref{eq:bilevel_mpec_formulation_general_first_constraint}-\eqref{eq:bilevel_mpec_formulation_general_final_constraint}, the adjoint equation:
  \begin{multline} \label{eq: adjoint general}
    \nabla \Jcal(u^*) - \nabla_{uu}^2 \Dcal_0(u^*)^\intercal p - \sum_{j=1}^K \sum_{i=1}^{k_j} P_j(\lambda^*_j)_i \nabla_{uu}^2 (\Dcal_j(u^*)_i)^\intercal p - \sum_{j=1}^M \Kbb_j^\star \pi_j - \sum_{j=1}^N \Ebb_j^\star \zeta_j\\
    + \sum_{j=1}^L \Bbold_j^\intercal \diag(-c^*_j - R_j(\sigma^*_j))  \mu_{\Bbold,j}^-  + \sum_{j=1}^L \Bbold_j^\intercal \diag(-c^*_j+R_j(\sigma^*_j))  \mu_{\Bbold,j}^+ - \sum_{j=1}^L \Bbold_j^\intercal \diag(c^*_j) \mu_{c,j} =0,
  \end{multline}
the relation between the adjoint state $p$ and the dual variables' multipliers:
  \begin{align}
      &- \Bbold_j p - (\Bbold_j u^*)\circ (\mu_{\Bbold,j}^- + \mu_{\Bbold,j}^+)- \mu_{\sigma,j}^- + \mu_{\sigma,j}^+ - (\Bbold_j u^*)\circ \mu_{c,j} =0, && \forall j=1, \dots,L,\\
      & \vartheta_j = \Kbb_j p , && \forall j=1, \dots, M,\\
      & \eta_j = \Ebb_j p , && \forall j=1, \dots,N,
  \end{align}
the gradient type equations:
  \begin{align}
      & -\nabla_\lambda P_j(\lambda^*_j) \nabla_u \Dcal_j (u^*)^\intercal p + \nabla_\lambda P_j(\lambda^*_j) \mu_{P,j} =0, && \forall j=1, \dots,K,\\
      & \nabla_\sigma R_j(\sigma^*_j) \diag(\Bbold_j u^*) (\mu_{\Bbold,j}^+ - \mu_{\Bbold,j}^-) - \nabla_\sigma R_j(\sigma^*_j) (\mu_{\sigma,j}^- + \mu_{\sigma,j}^+ + \mu_{R,j} )=0, && \forall j=1, \dots, L,\\
      & \nabla_\alpha Q_j(\alpha^*_j) \nu_j^{\alpha} =0, && \forall j=1, \dots, M,\\
      & \nabla_\beta S_j(\beta^*_j) \nu_j^{\beta} =0, && \forall j=1, \dots,N,
    \end{align}
    the relations between the auxiliar constraints' multipliers:
    \begin{align}
      & \Cos(\theta^*_j) \pi_{j}^{x}+ \Sin(\theta^*_j) \pi_{j}^{y} - \gamma_j^{\alpha}=0,  && \forall j=1, \dots,M,\\
      & \Cos(\theta^*_j) \vartheta_j^{x}+ \Sin(\theta^*_j) \vartheta_j^{y} + \mu_{\delta,j} - \nu_j^{\alpha}=0,  && \forall j=1, \dots,M,\\
      & \Sin(\phi^*_j) \left( \Cos(\varphi^*_j) \zeta_j^x+ \Sin(\varphi^*_j) \zeta_j^y \right) + \Cos(\phi^*_j) \zeta_j^z - \gamma_j^{\beta}=0,  && \forall j=1, \dots,N,\\
      & - \Sin(\phi^*_j) \left( \Cos(\varphi^*_j) \eta_{j}^x+ \Sin(\varphi^*_j) \eta_{j}^y \right) - \Cos(\phi^*_j) \eta_{j}^z - \mu_{\tau,j} +\nu_j^{\beta}=0,  && \forall j=1, \dots,N,\\
      & - r^*_j \circ \left( \Sin(\theta^*_j) \pi_j^x - \Cos(\theta^*_j) \pi_j^y \right)
       + \delta^*_j \circ \left( \Sin(\theta^*_j) \vartheta_j^x- \Cos(\theta^*_j) \vartheta_j^y \right) =0,  && \forall j=1, \dots,M,\\
      & \rho^*_j \circ \Cos(\phi^*_j) \left( \Cos(\varphi^*_j) \zeta_j^x + \Sin(\varphi^*_j) \zeta_j^y \right)- \rho^*_j \circ \Sin(\phi^*_j) \zeta_j^z\\
      & \hspace{0.1cm} - \tau^*_j \circ \Cos(\phi^*_j) \left( \Cos(\varphi^*_j) \eta_j^x  + \Sin(\varphi^*_j) \eta_j^y \right) + \tau^*_j \circ \Sin(\phi^*_j) \eta_j^z =0,  && \forall j=1, \dots,N, \nonumber\\
      & -\rho^*_j \circ \Sin(\phi^*_j) \left( \Sin(\varphi^*_j) \zeta_j^x - \Cos(\varphi^*_j) \zeta_j^y \right) \\
      & \hspace{0.5cm} + \tau^*_j \circ \Sin(\phi^*_j) \left( \Sin(\varphi^*_j) \eta_j^x - \Cos(\varphi^*_j) \eta_j^y \right)=0,  && \forall j=1, \dots,N, \nonumber
  \end{align}
  the complementarity conditions for the absolute value terms:
  \begin{align}
    & 0 \geq (-c^*_{j} - R_j(\sigma^*_j)) \circ (\Bbold_j u^*) \perp \mu_{\Bbold,j}^- \geq 0, && \forall j=1, \dots,L,\\
    & 0 \geq (-c^*_{j} + R_j(\sigma^*_j)) \circ (\Bbold_j u^*) \perp \mu_{\Bbold,j}^+ \geq 0, && \forall j=1, \dots,L,\\
    & 0 \geq (-c^*_{j} - R_j(\sigma^*_j)) \perp \mu_{\sigma,j}^- \geq 0, && \forall j=1, \dots,L,\\
    & 0 \geq (-c^*_{j} + R_j(\sigma^*_j)) \perp \mu_{\sigma,j}^+ \geq 0, && \forall j=1, \dots,L,\\
    & 0 \leq c^*_{j} (\Bbold_j u^*) \circ  \perp \mu_{c,j} \geq 0, && \forall j=1, \dots,L,\\
    & 0 \leq R_j(\sigma^*_j) \perp \mu_{R,j} \geq 0, && \forall j=1, \dots,L,
  \end{align}
  the complementarity relations for the positivity constraints:
  \begin{align}
    & 0 \leq \delta^*_j \perp \mu_{\delta,j} \geq 0, && \forall j=1, \dots,M,\\
    & 0 \leq \tau^*_j \perp \mu_{\tau,j} \geq 0, && \forall j=1, \dots,N,\\
    & 0 \leq P_j(\lambda^*_j) \perp \mu_{P,j} \geq 0, && \forall j=1, \dots,K,
  \end{align}
  and the M-stationarity conditions
  \begin{align}
    & \nu_{j,i}^\alpha =0, && \forall i \in \Acal_j^\alpha (x^*), \quad \forall j=1, \dots,M\\
    & \gamma_{j,i}^\alpha =0, && \forall i \in \Ical_j^\alpha (x^*), \quad \forall j=1, \dots,M\\
    & \gamma_{j,i}^\alpha \nu_{j,i}^\alpha =0 \lor \gamma_{j,i}^\alpha \geq 0, \nu_{j,i}^\alpha \geq 0, && \forall i \in \Bcal_j^\alpha (x^*), \quad \forall j=1, \dots,M \label{eq: M stationarity biactive general problem 1}\\
    & \nu_{j,i}^\beta =0, && \forall i \in \Acal_j^\beta (x^*), \quad \forall j=1, \dots,N\\
    & \gamma_{j,i}^\beta =0, && \forall i \in \Ical_j^\beta (x^*), \quad \forall j=1, \dots,N\\
    & \gamma_{j,i}^\beta \nu_{j,i}^\beta =0 \lor \gamma_{j,i}^\beta \geq 0, \nu_{j,i}^\beta \geq 0, && \forall i \in \Bcal_j^\beta (x^*), \quad \forall j=1, \dots,N. \label{eq: M stationarity biactive general problem 2}
  \end{align}
  hold.
\end{theorem}
Due to its length, the proof of this theorem is provided in the supplementary material accompanying the article. Let us also notice that the hypotheses of \Cref{thm: strongg stationarity general problem} are stronger than the total variation counterpart. This is only for presentation purposes, as the proof is already very long. However, the hypotheses may be relaxed mimicking the ones in \Cref{thm: strong stationarity MPCC}.

\section{Conclusions and Perspectives}\label{sec:conclusions}
In this article we propose a reformulation of a family of bilevel imaging learning problems as mathematical programs with complementarity constraints (MPCC). This reformulation is based on a lifting of the primal-dual system, arising as necessary and sufficient optimality condition for the lower-level problem, through the introduction of trigonometric auxiliar variables. Thanks to the interpretation of this class of problems as MPCC, we are able to apply important tools from MPCC theory and obtain first-order necessary conditions and second-order sufficient optimality conditions that characterize M- and S-stationary points.

Furthermore, the proposed reformulation opens the door to the use of efficient solution algorithms to compute optimal parameters for the type of problems considered. In an accompanying article, we perform an exhaustive numerical study for various relevant imaging applications, using specialized nonlinear programming software, which is carefully adjusted to take advantage of the bilevel structure.

Although the focus of these two articles is on imaging applications, the studied reformulation also makes it possible to deal with other types of problems where nonsmooth sparsity-based regularizers are used. This occurs, for instance, in different inverse problems in which the presence of non-smooth regularization terms allows obtaining solutions that are more in line with available a-priori information about them (see, e.g., \cite{chan2003identification}). The extension is, however, not direct and our intention is to work along this direction in the future.

\printbibliography
\end{document}